\documentclass{amsart}

\usepackage{amsfonts}
\usepackage{txfonts}
\usepackage{amssymb}
\usepackage{graphicx}
\graphicspath{{figures/}}
\usepackage{epsfig,mathrsfs}
\usepackage[bf,SL,BF]{subfigure}
\usepackage[hypertexnames=false]{hyperref}
\usepackage{longtable}
\usepackage{booktabs}
\usepackage{cite}
\usepackage{diagbox}
\usepackage{caption}
\usepackage{float}
\usepackage{epstopdf}
\usepackage{latexsym, bm}
\usepackage{bbm}
\usepackage{algorithm, algorithmic}
\usepackage{cleveref}
\newtheorem{theorem}{Theorem}[section]

\newtheorem{lemma}[theorem]{Lemma}
\newtheorem{assumption}[theorem]{Assumption}
\newtheorem{prop}[theorem]{Proposition}
\theoremstyle{definition}
\newtheorem{definition}[theorem]{Definition}
\newtheorem{remark}[theorem]{Remark}
\numberwithin{equation}{section}
\numberwithin{figure}{section}
\numberwithin{table}{section}

\usepackage{latexsym, bm}
\begin{document}

	\title[Multiple fractional Laplacian Schr\"{o}dinger equation with drift]{The Calder\'on problem for the multi-fractional Laplacian Schr\"{o}dinger equation with drift terms}
	
	\author[Jin Wen, Yun-Long Liu, Jun-Liang Fu]{Jin Wen$^{1, 2}$, Yun-Long Liu$^{1}$, Jun-Liang Fu$^{1}$\\ \small \it 1. Department of Mathematics, Northwest Normal University, Lanzhou 730070, P. R. China\\
2. Gansu Prov Res Ctr Basic Disciplines Math \& Stat, Lanzhou 730070, P. R.China }
	\date{\today}
	\thanks{{*}Corresponding author. Email: wenj@nwnu.edu.cn. The work described in this article was supported by the NNSF of China (12261082, 11326234, 12501593), NSF of Gansu Province (145RJZA099), Scientific research project of Higher School in Gansu Province (2014A-012), Project of NWNU-LKQN2020-08, Graduate First Class Course Construction Project of Northwest Normal University in 2024 (2024KCLX01012) and Graduate Education and Curriculum Reform Project of Northwest Normal University in 2024 (2024KGLX01016).}

\begin{abstract}
This paper studies the Calder\'on inverse problem for the multi-fractional Laplacian Schr\"odinger equation with drift terms, whose leading operator is a sum of fractional Laplacians of different orders, perturbed by a drift vector field and a potential. We consider two inverse settings: in the homogeneous source case \(f=0\), we aim to simultaneously recover the drift coefficient \(b\) and the potential \(q\) from the associated Dirichlet-to-Neumann map; in the non-homogeneous case \(f\neq 0\), we aim to simultaneously determine \(b\), \(q\), and the source term \(f\). Under suitable assumptions, we establish the well-posedness of the forward problem, derive a strong unique continuation principle, prove Runge-type approximation properties in fractional Sobolev spaces and weighted \(C^k\) spaces, and obtain generalized Alessandrini identities. These tools yield global uniqueness results for the inverse problems. In addition, we develop a numerical reconstruction framework combining finite difference discretization, radial basis function surrogate modeling, and hierarchical Bayesian inversion. Numerical examples in one dimension for smooth and discontinuous coefficients and sources illustrate the effectiveness of the proposed method and support the theoretical uniqueness results.

\end{abstract}

	\maketitle

\section{Introduction}
In 1980, Calder\'on proposed the classic Calder\'on inverse problem \cite{Cal-1980}. Since then, it has been extended to nonlocal Calder\'on problems. The physical motivation for nonlocal Calder\'on problems comes from the ability of fractional calculus to describe nonlocality, memory effects, and anomalous diffusion. Fractional models have a wide range of applications in quantum mechanics, medical imaging, geophysical exploration, and environmental science \cite{LPG-2020, QU-2024, KLW-2022}.

Regarding the research on nonlocal Calder\'on problems, T. Ghosh et al. first proposed the study of the Calder\'on problem in the fractional Schr\"odinger equation \cite{GSU-2016}. They proved the global uniqueness of the inverse problem in two specific cases, that is, the unknown potential function within the boundary can be uniquely determined by the external measurement of the solution. For the generalized fractional-type Calder\'on problem, many scholars have conducted research in recent years and achieved fruitful results. For example, Fractional inverse Laplace problem under low-order nonlocal perturbations \cite{Ghosh2020}, Uniqueness and Reconstruction of fractional Calder\'on problem under single measurement \cite{BGU-2021}, uniqueness of higher-order Calder\'on problem involving linear local operators \cite{CMRU-2022}, and Calder\'on problem for the fractional Dirac operator \cite{QU-2024}. For more research on fractional Calder\'on  problems, please refer to \cite{CR-2021,Fei-2024,BS-2024}.

In \cite{CekicLinRuland2020}, the authors study the Calder\'on problem for the fractional Schr\"odinger equation with a drift term, proving that the unknown drift and potential functions in a bounded domain can be simultaneously and uniquely determined by an infinite number of external measurements. The uniqueness result is also supported by relevant logarithmic stability estimates under appropriate a priori assumptions. Additionally, it discusses the existence of uniqueness under a finite number of general measurements. In \cite{LinLiuLo2025}, the authors study the inverse problem involving multi-fractional operators and proposes a class of equations involving multi-fractional operators. The unique continuation property for fractional harmonic functions is also presented. This article focuses on inverse problem for the multi-fractional Schr\"odinger equation with drift, which is an extension of their work.

In this paper, we consider the following multi-fractional Laplacian Schr\"{o}dinger equation with drift terms
\begin{equation}
\label{eq11}
\begin{cases}
(\sum_{k=1}^m (-\Delta)^{s_k} + b(x)\cdot\nabla + q(x))u(x) = f(x), & x \in \Omega, \\
u(x) = g(x), & x \in \Omega_e = \mathbb{R}^n \setminus \overline{\Omega},
\end{cases}
\end{equation}
where $\Omega \subset \mathbb{R}^n$ be a non-empty bounded domain (open and connected), $s_1 < \cdots < s_m,\ s_m>\frac12 $, \(b\in W^{1-s_m,\infty}(\Omega)^n\), \(q\in L^\infty(\Omega)\), and \(f\in L^2(\Omega)\).
The fractional Laplacian $(-\Delta)^s$ is defined via the hypersingular integral
\begin{equation}
\label{eq12}
(-\Delta)^s u(x) = c_{n,s} \, \mathrm{P.V.} \int_{\mathbb{R}^n} \frac{u(x) - u(\mathbf{y})}{|x - \mathbf{y}|^{n+2s}} \, d\mathbf{y},
\end{equation}
with the normalization constant
\begin{equation}\label{eq13}
c_{n,s} = \frac{2^{2s} s \, \Gamma(\frac{n}{2} + s)}{\pi^{n/2} \Gamma(1 - s)}.
\end{equation}
In this paper, we consider the following two problems:
\begin{itemize}
\item[(A)]: If \(f=0\), we aim to simultaneously recover the drift coefficient $b$ and the potential $q$ simultaneously from the relevant DTN map.
\item[(B)]: When \(f\neq 0\), we wish to recover simultaneously the drift coefficient $b$, the potential function $q$, and the source term $f$ from the associated DtN map.
\end{itemize}
Regarding these two problems, the main contributions of this paper are threefold. Firstly, when $f = 0$, we proved the uniqueness of the simultaneous inversion of the drift $b$ and the potential $q$. Secondly, when $f\neq0$, we proved that the complete Dirichlet-to-Neumann (DtN) mapping uniquely determines $b$, $q$ and $f$. Thirdly, we established the required Runge approximation property and the generalized Alessandrini identity. At the same time, we presented a numerical method and numerical experiments were provided to illustrate the reconstruction effect.

The qualitative Runge approximation property is a key tool for studying fractional-order Schr\"odinger-type inverse problems, which was first proposed by \cite{GSU-2016}. To derive this result, the authors of \cite{GSU-2016} used the unique continuation property obtained from Carleman estimates for the fractional Schr\"odinger equation in \cite{Ruland-2015}, and also referred to the studies in \cite{FF-2014, FF-2015, Seo-2015, R-2017, Ghosh2020, Yu-2017} on the unique continuation properties related to the fractional-order Schr\"odinger equation. For the fractional-order Schr\"odinger operator with variable coefficients, in \cite{GLX-2017} the authors  used the Almgren frequency function to derive this property.

To prove the approximation property, we need to use the unique continuation. We have noticed that there are significant differences between the classical Calder\'on problem and the fractional-order Calder\'on problem. Specifically, when dealing with the fractional-order problem, it is not necessary to construct the CGO solution. On the other hand, unique continuation is a notable feature of fractional-order operators, while local operators do not possess this property. This makes the fractional-order inverse problem easier to handle and helps us obtain stronger results. In fact, \cite{CGRU-2026} has pointed out that the uniqueness results obtained in the local case can guarantee their validity in the fractional-order case, but not vice versa. Therefore, when studying high-order fractional Laplacian operators, many works focus on deriving the unique continuation principle, such as \cite{Seo-2014, Seo-20152, Seo-2015, GAR-2019, FeF-2020, CMRU-2022, KRZ-2023, SA-2020}. In these studies, many rely on the extension of the Caffarelli-Silvestre type, which has been extended to higher-order fractional exponents.

The following condition will be assumed in this paper:
\begin{assumption}
\label{ass1}
If \(u\in H^{s_m}(\mathbb{R}^{n})\) solves
\begin{equation*}
\left\{
\begin{array}{ll}
(\sum_{k=1}^m (-\Delta)^{s_k} + b(x)\cdot\nabla + q(x))u(x) = 0, & x \in \Omega, \\
u(x) = 0, & x \in \Omega_e,
\end{array}
\right.
\end{equation*}
then \(u\equiv 0\). This is equivalent to the assumption that zero is not a Dirichlet eigenvalue of \(\sum_{k=1}^m (-\Delta)^{s_k} + b\cdot\nabla + q\).
\end{assumption}

This paper is organized as follows. Section~\ref{sec2} introduces the preliminary framework for the multi-fractional Laplacian Schr\"odinger equation with drift, including fractional Sobolev spaces, the bilinear form and well-posedness of the forward problem, the DtN map, and the strong unique continuation property. Section~\ref{sec3} develops the Runge approximation property, first in $\widetilde H^{s_m}(\Omega)$ and then in higher-order weighted $C^k$ spaces. Section~\ref{sec4} studies uniqueness, generalized Alessandrini identities are derived, and it is shown that the DtN map uniquely determines $(b,q)$ in the homogeneous case $f=0$ and $(b,q,f)$ in the non-homogeneous case $f\neq0$. Section~\ref{sec5} presents the numerical method, including parameterization, finite difference discretization of the forward problem, the RBF surrogate model, and a hierarchical Bayesian inversion framework with ARD/TV-GSM priors, IAS optimization, and uncertainty quantification. Section~\ref{sec6} reports four one-dimensional numerical examples for smooth and discontinuous coefficients in both homogeneous and non-homogeneous settings, with different noise levels. Finally, Section~\ref{sec7} provides a short conclusion.

\section{Preliminaries}\label{sec2}

Let $n\in\mathbb N$ and let $\mathbb R^n$ be the $n$-dimensional Euclidean space. We denote by $\mathcal S(\mathbb R^n)$ the Schwartz space and by $\mathcal S'(\mathbb R^n)$ its dual space of tempered distributions. The Fourier transform of a tempered distribution $u$ is defined by
\begin{equation*}
\widehat u(\xi)=\int_{\mathbb R^n} e^{-2\pi i x\cdot \xi}u(x)\,dx,
\end{equation*}
with the usual extension to $\mathcal S'(\mathbb R^n)$.

For $s\in\mathbb R$, the fractional Sobolev space $H^s(\mathbb R^n)$ is defined by
\begin{equation*}
H^s(\mathbb R^n):=\Bigl\{u\in\mathcal S'(\mathbb R^n):
\|u\|_{H^s(\mathbb R^n)}:=\bigl\|(1+|\xi|^2)^{s/2}\widehat u(\xi)\bigr\|_{L^2(\mathbb R^n)}<\infty\Bigr\}.
\end{equation*}
The homogeneous space $\dot H^s(\mathbb R^n)$ is given by
\begin{equation*}
\dot H^s(\mathbb R^n):=\Bigl\{u\in\mathcal S'(\mathbb R^n):
\|u\|_{\dot H^s(\mathbb R^n)}:=\bigl\||\xi|^s\widehat u(\xi)\bigr\|_{L^2(\mathbb R^n)}<\infty\Bigr\}.
\end{equation*}
For $s>0$, the norms $\|\cdot\|_{\dot H^s}$ and $\|(-\Delta)^{s/2}\cdot\|_{L^2}$ are equivalent.

For $s>0$, the dual space $H^{-s}(\mathbb R^n)$ is defined by the norm
\begin{equation*}
\|u\|_{H^{-s}(\mathbb R^n)}:=\bigl\|(1+|\xi|^2)^{-s/2}\widehat u(\xi)\bigr\|_{L^2(\mathbb R^n)}.
\end{equation*}

Let $\mathcal O\subset\mathbb R^n$ be a nonempty open set. We define the following fractional Sobolev spaces on $\mathcal O$:
\begin{align*}
H^s(\mathcal O)&:=\bigl\{u|_{\mathcal O}: u\in H^s(\mathbb R^n)\bigr\},\\
\widetilde H^s(\mathcal O)&:=\overline{C_c^\infty(\mathcal O)}^{\,\|\cdot\|_{H^s(\mathbb R^n)}},\\
H_0^s(\mathcal O)&:=\overline{C_c^\infty(\mathcal O)}^{\,\|\cdot\|_{H^s(\mathcal O)}},
\end{align*}
and
$$H_{\overline{\mathcal O}}^s:=\bigl\{u\in H^s(\mathbb R^n): \operatorname{supp}u\subset \overline{\mathcal O}\bigr\}.$$
The space $H^s(\mathcal O)$ is equipped with the quotient norm
\begin{equation*}
\|u\|_{H^s(\mathcal O)}:=\inf\bigl\{\|v\|_{H^s(\mathbb R^n)}: v\in H^s(\mathbb R^n),\ v|_{\mathcal O}=u\bigr\}.
\end{equation*}

For a bounded Lipschitz domain $\Omega$, we have
\begin{equation*}
\widetilde H^s(\Omega) = H_{\overline{\Omega}}^s \quad \text{for } s>-1/2,\ s\notin\{1/2,3/2,\dots\},
\end{equation*}
and the duality relations
\begin{equation*}
\bigl(H^s(\Omega)\bigr)^* = \widetilde H^{-s}(\Omega),\qquad
\bigl(\widetilde H^s(\Omega)\bigr)^* = H^{-s}(\Omega).
\end{equation*}

For the treatment of drift terms, we also use the $L^p$-based fractional Sobolev spaces $W^{s,p}(\mathbb R^n)$, defined for $1<p<\infty$ and $0<s<1$ by
\begin{equation*}
W^{s,p}(\mathbb R^n):=\bigl\{u\in L^p(\mathbb R^n): \|u\|_{W^{s,p}(\mathbb R^n)}:=\|(1-\Delta)^{s/2}u\|_{L^p(\mathbb R^n)}<\infty\bigr\}.
\end{equation*}

\subsection{Multi-fractional Laplacian}

Let $m\in\mathbb N$, and let
\begin{equation*}
0<s_1<s_2<\cdots<s_m<1,\qquad s_m>\frac{1}{2}.
\end{equation*}
For each $i=1,\dots,m$, we define the \emph{multi-fractional operator}
\begin{equation*}
P((-\Delta)^s)u:=\sum_{i=1}^m  (-\Delta)^{s_i}u.
\end{equation*}

We shall use the following standard estimates.

\begin{lemma}[Poincar\'e inequality \cite{CoviMonkkonenRailo2021}]\label{lm21}
Let $s>0$, $s\notin\mathbb N$, $K\subset\mathbb R^n$ compact, and $u\in H_K^s(\mathbb R^n)$. Then there exists a constant $C=C(n,s,K)>0$ such that
\begin{equation*}
\|u\|_{L^2(\mathbb R^n)}\le C\|(-\Delta)^{s/2}u\|_{L^2(\mathbb R^n)}.
\end{equation*}
\end{lemma}

\begin{lemma}[Sobolev-type inequality \cite{RailoZimmermann2022}]\label{lm22}
Let $\Omega\subset\mathbb R^n$ be bounded open, $0\le r\le s$, and $u\in\widetilde H^s(\Omega)$. Then there exists $C=C(n,r,s,\Omega)>0$ such that
\begin{equation*}
\|(-\Delta)^{r/2}u\|_{L^2(\mathbb R^n)}\le C\|(-\Delta)^{s/2}u\|_{L^2(\mathbb R^n)}.
\end{equation*}
\end{lemma}

Throughout the paper, $\Omega\subset\mathbb R^n$ denotes a bounded Lipschitz domain, and its exterior is denoted by
\begin{equation*}
\Omega_e:=\mathbb R^n\setminus\overline{\Omega}.
\end{equation*}
For an arbitrary nonempty open set $W\subset\Omega_e$, we may define boundary measurements supported in $W$.

\subsection{Bilinear form and well-posedness}
For $F \in (\widetilde{H}^{s_m}(\Omega))^*$ (the dual space of $H^{s_m}(\Omega)$), $g \in H^{s_m}(\mathbb{R}^n)$, let us consider the following Dirichlet problem
\begin{equation}
\label{eq21}
\begin{aligned}
(\sum_{k=1}^m (-\Delta)^{s_k} + b\cdot\nabla + q)u &= f\  \text{in}\ \Omega,\\
&u - g \in \widetilde{H}^s(\Omega).
\end{aligned}
\end{equation}

Given an arbitrary non-empty open set $\mathcal{O} \subset \mathbb{R}^n$ and $v,w \in L^2(\mathcal{O})$, we use the notation
\begin{equation*}
(v,w)_{\mathcal{O}} := \int_{\mathcal{O}} v w \, dx.
\end{equation*}

For $v,w \in C_c^\infty(\mathbb{R}^n)$ we define the bilinear form $B_{b,q}(\cdot,\cdot)$ by
\begin{equation}
\label{eq22}
B_{b,q}(v,w) := \sum_{k=1}^m ((-\Delta)^{s_k/2} v, (-\Delta)^{s_k/2} w)_{\mathbb{R}^n}
+ (b\cdot \nabla v, w)_\Omega + (q v, w)_\Omega.
\end{equation}

Meanwhile, we introduce the adjoint bilinear form as
\begin{equation}
\label{eq23}
B_{b,q}^*(v^*, w^*) := \sum_{k=1}^m ((-\Delta)^{s_k/2} v^*, (-\Delta)^{s_k/2} w^*)_{\mathbb{R}^n}+ (b v^*, \nabla w^*)_\Omega + (q v^*, w^*)_\Omega,
\end{equation}
for $v^*, w^* \in C_c^\infty(\mathbb{R}^n)$.

\begin{lemma}[Boundedness]\label{lm23}
For $n \ge 1$, let $\Omega \subset \mathbb{R}^n$ be a bounded Lipschitz domain and $\frac12 < s_m < 1$. Assume that $b \in W^{1-s_m,\infty}(\Omega)^n$, $q \in L^\infty(\Omega)$. Let $B_{b,q}(\cdot,\cdot)$ be the bilinear and the adjoint bilinear form defined by \eqref{eq22}, respectively. Then, $B_{b,q}(\cdot,\cdot)$ extend as bounded bilinear forms on $H^{s_m}(\mathbb{R}^n) \times H^{s_m}(\mathbb{R}^n)$.
\end{lemma}
\begin{proof}
See the Appendix.
\end{proof}

Based on the above definition, \eqref{eq22}, \eqref{eq23}, as well as Lemma \ref{lm23}, we can obtain
\begin{equation}
\label{eq28}
B_{b,q}(v,w)=B_{b,q}^*(w,v),\ \text{for any}\ v,\ w\in H^{s_m}(R^n).
\end{equation}

\begin{lemma}[G\aa rding inequality]\label{lm24}
There exist constants $c_0>0$ and $\mu\ge 0$ such that for all $u \in \widetilde H^{s_m}(\Omega)$,
\begin{equation*}
B_{b,q}(u,u) + \mu \|u\|_{L^2(\mathbb R^n)}^2 \ge c_0 \|u\|_{H^{s_m}(\mathbb R^n)}^2.
\end{equation*}
\end{lemma}

\begin{proof}
See the Appendix.
\end{proof}

The following proposition establishes the well-posedness of the Dirichlet problem for the multi-fractional operator with drift, up to a discrete set of exceptional frequencies.

\begin{lemma}\label{prop:wellposed}
Let $\Omega\subset\mathbb R^n$ be a bounded Lipschitz domain, $n\ge1$, $0<s_1<\cdots<s_m$ with $s_m>\frac12$, $b\in W^{1-s_m,\infty}(\Omega)^n$, and $q\in L^\infty(\Omega)$. There exists a countable set $\Sigma\subset\mathbb R$ (the set of Dirichlet eigenvalues of the operator $\sum_{k=1}^m(-\Delta)^{s_k}+b\cdot\nabla+q$) such that the following hold:

(a) For every $\lambda\notin\Sigma$, every $F\in(\widetilde H^{s_m}(\Omega))^*$, and every $g\in H^{s_m}(\mathbb R^n)$, there exists a unique $u\in H^{s_m}(\mathbb R^n)$ satisfying
\begin{equation*}
B_{b,q}(u,w)-\lambda (u,w)_\Omega = F(w) \quad\forall\, w\in\widetilde H^{s_m}(\Omega),\qquad u-g\in\widetilde H^{s_m}(\Omega).
\end{equation*}
Moreover, there is a constant $C>0$, independent of $F$ and $f$, such that
\begin{equation*}
\|u\|_{H^{s_m}(\mathbb R^n)} \le C\left(\|F\|_{(\widetilde H^{s_m}(\Omega))^*}+\|f\|_{H^{s_m}(\mathbb R^n)}\right).
\end{equation*}

(b) The same statement holds for the adjoint problem with the bilinear form $B_{b,q}^*$; the exceptional set is the same $\Sigma$.

(c) If $\lambda=0\notin\Sigma$ (i.e., Assumption \ref{ass1} holds), then for every $F$ and $g$ the solutions exist uniquely; in particular, the homogeneous equation with zero exterior data has only the trivial solution.
\end{lemma}

\begin{proof}
See the Appendix.
\end{proof}

By Lemma~\ref{prop:wellposed}, for each exterior datum \(g\in H^{s_m}(\mathbb R^n)\) there exists a unique solution \(u_g\) of the homogeneous problem
\begin{equation}
\label{eq29}
\begin{cases}
(\sum_{k=1}^m (-\Delta)^{s_k} + b\cdot\nabla + q) u = 0 & \text{in } \Omega,\\
u = g & \text{in } \Omega_e.
\end{cases}
\end{equation}
For the non-homogeneous case \(f\in L^2(\Omega)\), the same proposition (with \(F(w)=\int_\Omega f w\,dx\)) ensures a unique solution \(u\) of
\begin{equation}
\label{eq210}
\begin{cases}
(\sum_{k=1}^m (-\Delta)^{s_k} + b\cdot\nabla + q) u = f & \text{in } \Omega,\\
u = g & \text{in } \Omega_e.
\end{cases}
\end{equation}
By linearity, \(u = u_g + u_f\), where \(u_f\) solves the zero-exterior-data problem
\begin{equation}
\label{eq211}
\begin{cases}
(\sum_{k=1}^m (-\Delta)^{s_k} + b\cdot\nabla + q)u_f = f & \text{in } \Omega,\\
u_f = 0 & \text{in } \Omega_e.
\end{cases}
\end{equation}

\begin{definition}[DtN map]
\label{def1}
Let $\Omega \subset \mathbb{R}^n$, $n \geq 1$, be a bounded Lipschitz domain, let $s_m>\frac{1}{2}$ and let $b \in W^{1-s_m,\infty}(\Omega)^n$, $q \in L^\infty(\Omega)$.
\begin{equation*}
\langle \Lambda_{b,q} f,\, g\rangle_{X^*\times X}
:= B_{b,q}(u_g,\, h),\qquad  g,  h\in  H^{s_m}(\mathbb{R}^n) ,
\end{equation*}
where $X$ is the abstract trace space $X = H^{s_m}(\mathbb{R}^n) / \widetilde H^{s_m}(\Omega)$, $u_g\in H^{s_m}(\mathbb{R}^n)$ solves $((-\Delta)^s+b\nabla + q)u = 0$ in $\Omega$ with $u-f \in \widetilde{H}^s(\Omega)$.
\end{definition}

\begin{remark}In addition, we use the definition \ref{def1} can be calculated
\begin{equation}
\label{eq212}
\begin{aligned}
(\Lambda_{b,q}[g],[h]) &= B_{b,q}(u_g,h)\\
&= \int_{\mathbb{R}^n} \sum_{k=1}^{m}(-\Delta)^{s_k/2}u_g (-\Delta)^{s_k/2}h \, dx
+ \int_{\Omega} b\cdot \nabla u_g \, h \, dx
 + \int_{\Omega} q\, u_g \, h \, dx\\
&= \int_{\Omega_e} h \sum_{k=1}^{m}(-\Delta)^{s_k} u_g \, dx,
\end{aligned}
\end{equation}
Then from \eqref{eq212}, one can formally obtain that
\begin{equation*}
\Lambda_{b,q}[g] =  \sum_{k=1}^{m}(-\Delta)^{s_k}u_g\big|_{\Omega_e}.
\end{equation*}
 We remark that in order to make this calculation rigorous, higher regularity assumptions have to be imposed. We refer to Section 3, Lemma 3.1, in [20] for details on this.
\end{remark}

Next, we introduce the adjoint operator
\begin{equation}
\label{eq213}
\begin{aligned}
\sum_{k=1}^{m}(-\Delta)^{s_k} u^* - \nabla\cdot (b u^*) + &q u^* = 0\  \text{in}\ \Omega,\\
&u^* - g \in \widetilde{H}^{s_m}(\Omega),
\end{aligned}
\end{equation}
we can define the DtN map \(\Lambda_{b,q}^*\) of the adjoint bilinear form \eqref{eq213}, i.e.,
\begin{equation*}
\Lambda_{b,q}^* : X \to X^*, \qquad (\Lambda_{b,q}^*[g],[h]) = B_{b,q}^*(u_g^*, h).
\end{equation*}
Here \(u_g^*\in H^{s_m}(\mathbb{R})^n\) is a weak solution of \eqref{eq213} with exterior data \(g\).

Let \(\Lambda_{b,q}^*\) denote the DtN map associated with the adjoint equation \(L_{b,q}^*v=0\) in \(\Omega\), defined analogously via \(B_{b,q}^*\).

\begin{lemma}\label{lm25}
Let \(\Omega \subset \mathbb{R}^n\), \(n\geq 1\), be a bounded, non-empty open Lipschitz set, \(s_m>\frac12\), \(b\in W^{1-s_m,\infty}(\Omega)^n\), \(q\in L^\infty(\Omega)\) and assumption \ref{ass1} holds. Let \(B_{b,q}(\cdot,\cdot)\), \(B_{b,q}^*(\cdot,\cdot)\) and \(\Lambda_{b,q}\), \(\Lambda_{b,q}^*\) be given as above. Then, for all \(g,h\in X\)
\begin{equation*}
\langle \Lambda_{b,q} g,\, h\rangle
= \langle g,\, \Lambda_{b,q}^* h\rangle.
\end{equation*}
\end{lemma}

\begin{proof}
Let \(u_h\) be the solution of \eqref{eq213}. Using the definition of the DtN maps and the symmetry of the bilinear forms, we obtain
\begin{equation*}
\langle \Lambda_{b,q} g,\, h\rangle
= B_{b,q}(u_g,\,u_h)
= B_{b,q}^*(u^*_h,\,u_g)
= \langle \Lambda_{b,q}^* h,\, g\rangle
= \langle g,\, \Lambda_{b,q}^* h\rangle,
\end{equation*}
where the last equality holds in the duality pairing \(X^*\times X\).
\end{proof}
Next, we present the strong unique continuation principle for homogeneous multi-fractional operators.
\begin{lemma}[UCP for multi-fractional operators \cite{LinLiuLo2025}]
\label{lm26}
Let \(u \in H^{s_m}(\mathbb{R}^n) \cap C_c^\infty(\mathbb{R}^n)\). Assume that for some non-empty open set $W\subset R^n$ we have
\begin{equation*}
u =0\ \text{and}\ \sum_{k=1}^m (-\Delta)^{s_k} u = 0,\quad u \ \text{in } W,
\end{equation*}
then \(u \equiv 0\) in \(\mathbb{R}^n\).
\end{lemma}

\section{Approximation property}\label{sec3}
By using the above strong uniqueness property Lemma \ref{lm26} and a duality argument, we can derive the following Runge approximation property for solutions of the multi-fractional Schr\"odinger equation.

\begin{theorem}
\label{thm:runge}
For $n \ge 1$ and $s_m>\frac12$, let $\Omega \subset \mathbb{R}^n$ be a bounded open, non-empty Lipschitz set and $\Omega_1 \subset \mathbb{R}^n$ be an arbitrary open set such that $\Omega_1 \setminus \overline{\Omega} \neq \emptyset$.
(a) Let $b \in W^{1-s_m,\infty}(\Omega)^n$ and $q \in L^\infty(\Omega)$. Then, for any $h \in L^2(\Omega)$ and $\epsilon >0$, one can find a solution $u_\epsilon \in H^{s_m}(\mathbb{R}^n)$ of
\begin{equation*}
(\sum_{k=1}^m (-\Delta)^{s_k} + b\cdot \nabla + q)u_\epsilon = 0 \quad \text{in } \Omega,
\end{equation*}
with $\operatorname{supp}(u_\epsilon) \subset \overline{\Omega_1}$, such that
\begin{equation*}
\|u_\epsilon - h\|_{L^2(\Omega)} < \epsilon.
\end{equation*}
(b) If we further assume that $\Omega$ has a $C^\infty$-smooth boundary, $b \in C_c^\infty(\Omega)^n$ and $q \in C_c^\infty(\Omega)$ with $\operatorname{supp}(b), \operatorname{supp}(q) \Subset \Omega$, given any $h \in C^\infty(\overline{\Omega})$, $\epsilon >0$ and $k \in \mathbb{N}$, then there exists a solution $u_\epsilon \in H^{s_m}(\mathbb{R}^n)$ of
\begin{equation*}
(\sum_{k=1}^m (-\Delta)^{s_k} + b\cdot \nabla + q)u_\epsilon = 0 \quad \text{in } \Omega,
\end{equation*}
with $\operatorname{supp}(u_\epsilon) \subset \overline{\Omega_1}$, such that
\begin{equation*}
\| d^{-s}(x) u_\epsilon - h \|_{C^k(\Omega)} \le \epsilon.
\end{equation*}
The function $d(x)$ is any $C^\infty$-smooth function defined in $\overline{\Omega}$ such that $d>0$ in $\Omega$ and $d(x) = \operatorname{dist}(x,\partial\Omega)$ whenever $x$ is near $\partial\Omega$.
\end{theorem}

\subsection{Approximation in $\widetilde{H}^{s_m}(\Omega)$}
Let $\Omega \subset \mathbb{R}^n$, $n \geq 1$, be a bounded, non-empty open Lipschitz set. Let $s_m >\frac12$ and assume that \ref{ass1} holds. Then, the Poisson operator $P_{b,q}$ associated with \eqref{eq29} is defined as
\begin{equation}
\label{eq31}
P_{b,q} : X \to H^{s_m}(\mathbb{R}^n), \qquad g \mapsto u_g,
\end{equation}
where $u_g \in H^{s_m}(\mathbb{R}^n)$ with $u_g - g \in \widetilde{H}^{s_m}(\Omega)$ denotes the unique solution to \eqref{eq29}.

\begin{lemma}\label{lm32}
\label{lem:Runge}
For $n \geq 1$ and $s_m >\frac12$, $b \in W^{1-s_m,\infty}(\Omega)^n$ and $q \in L^\infty(\Omega)$. Further let $W \subset \Omega_e$ be an arbitrary non-empty open set. Let
\begin{equation*}
\mathcal{D} \coloneqq \left\{ P_{b,q} g - g : g \in C_c^\infty(W) \right\}.
\end{equation*}
Then we have the following results:\\
(a) The set $\mathcal{D}$ is dense in $L^2(\Omega)$.\\
(b) The set $\mathcal{D}$ is dense in $\widetilde{H}^{s_m}(\Omega)$.
\end{lemma}

\begin{proof}
See the Appendix.
\end{proof}

\subsection{Higher-order Runge approximation}

In this subsection we prove the higher-order Runge approximation property in weighted $C^k$ spaces, following the strategy of \cite{CekicLinRuland2020}. We work under stronger regularity assumptions on the coefficients and the domain. Throughout this subsection we assume that $\Omega\subset\mathbb R^n$ is a bounded domain with $C^\infty$ boundary, $0<s_1<\cdots<s_m<1$ with $s_m>\frac12$, and
\begin{equation*}
b\in C_c^\infty(\Omega)^n,\qquad q\in C_c^\infty(\Omega),\qquad \operatorname{supp}(b),\operatorname{supp}(q)\Subset\Omega,
\end{equation*}
and that Assumption~\ref{ass1} holds. We denote by
\begin{equation*}
L_{b,q} := \sum_{k=1}^m (-\Delta)^{s_k} + b\cdot\nabla + q
\end{equation*}
the multi-fractional Schr\"odinger operator with drift.

Let $d(x):=\operatorname{dist}(x,\partial\Omega)$ be a smooth positive function in $\Omega$ that equals the distance to the boundary near $\partial\Omega$. Define the weighted Fr\'echet space
\begin{equation*}
\mathcal E^{s_m}(\overline{\Omega}) := e_\Omega\, d(x)^{s_m} C^\infty(\overline{\Omega}),
\end{equation*}
where $e_\Omega$ denotes extension by zero from $\Omega$ to $\mathbb R^n$. Equivalently,
\begin{equation*}
\mathcal E^{s_m}(\overline{\Omega}) = \bigl\{ e_\Omega v : v\in C^\infty(\overline{\Omega}),\ d^{-s_m}v\in C^\infty(\overline{\Omega}) \bigr\}.
\end{equation*}
We equip it with the topology induced by the seminorms $\|d^{-s_m}w\|_{C^k(\overline\Omega)}$, $k\in\mathbb N$.

For $\mu>s_m-\frac12$ we work in the Banach space $H^{s_m(\mu)}(\overline{\Omega})$, which is tailored for solutions $u$ of
\begin{equation*}
r_\Omega L_{b,q}u \in H^{\mu-2s_m}(\Omega),\qquad u=0\ \text{in }\Omega_e,
\end{equation*}
where $r_\Omega$ is the restriction map from $\mathbb R^n$ to $\Omega$. The following result is the multi-fractional analogue of \cite[Lemma 7.1]{CekicLinRuland2020}; it follows from the regularity theory for the fractional Laplacian on smooth domains together with the fact that the highest-order term $(-\Delta)^{s_m}$ dominates the lower-order fractional terms and the drift.

\begin{prop}\cite{GSU-2016}\label{prop:Hsmu}
For $\mu>s_m-\frac12$ and a smooth bounded domain $\Omega\subset\mathbb R^n$ there exists a Banach space $H^{s_m(\mu)}(\overline{\Omega})$ with the following properties:
\begin{enumerate}
\item $H^{s_m(\mu)}(\overline{\Omega})\subset H_{\overline{\Omega}}^{s_m-\frac12}$ with a continuous inclusion;
\item $H^{s_m(\mu)}(\overline{\Omega})=H_{\overline{\Omega}}^{\mu}$ for $\mu\in(s_m-\frac12,s_m+\frac12)$;
\item $r_\Omega L_{b,q}$ is a homeomorphism from $H^{s_m(\mu)}(\overline{\Omega})$ onto $H^{\mu-2s_m}(\Omega)$;
\item $H_{\overline{\Omega}}^{\mu}\subset H^{s_m(\mu)}(\overline{\Omega})\subset H_{\mathrm{loc}}^{\mu}(\Omega)$ with continuous inclusions, or the multiplication by any smooth cut-off $\chi\in C_c^\infty(\Omega)$ is bounded from $H^{s_m(\mu)}(\overline{\Omega})$ to $H^{\mu}(\Omega)$;
\item $\mathcal E^{s_m}(\overline{\Omega})=\bigcap_{\mu>s_m-\frac12}H^{s_m(\mu)}(\overline{\Omega})$ and the set $\mathcal E^{s_m}(\overline{\Omega})$ is dense in $H^{s_m(\mu)}(\overline{\Omega})$.
\end{enumerate}
\end{prop}

Equipped with the topology induced by $\{\|\cdot\|_{H^{s_m(k)}}\}_{k=1}^{\infty}$, the space $\mathcal E^{s_m}(\overline{\Omega})$ is a Fr\'echet space. Building on these properties, we prove the density of the solution set in $\mathcal E^{s_m}(\overline{\Omega})$.

\begin{lemma}\label{lm34}
Let $W\subset\Omega_e$ be a nonempty open set. Let $P_{b,q}$ be the Poisson operator for the homogeneous problem
\begin{equation*}
L_{b,q}u=0\ \text{in }\Omega,\qquad u=g\ \text{in }\Omega_e,
\end{equation*}
and define
\begin{equation*}
\mathcal D := \bigl\{ e_\Omega(r_\Omega P_{b,q}g) : g\in C_c^\infty(W) \bigr\}.
\end{equation*}
Then $\mathcal D$ is dense in the Fr\'echet space $\mathcal E^{s_m}(\overline{\Omega})$ with the topology induced by $\{\|\cdot\|_{H^{s_m(k)}}\}_{k=1}^{\infty}$.
\end{lemma}

\begin{proof}
See the Appendix.
\end{proof}

Now we are ready to prove the higher-order Runge approximation property.

\begin{proof}[Proof of Theorem~\ref{thm:runge}(b)]
Since $\Omega\Subset\Omega_1$ with $\operatorname{int}(\Omega_1\setminus\overline{\Omega})\neq\emptyset$, we may choose a small ball $W$ with $\overline W\subset\Omega_1\setminus\overline{\Omega}$. Let $g\in C^\infty(\overline{\Omega})$ and define
\begin{equation*}
h := e_\Omega d(x)^{s_m} g \in \mathcal E^{s_m}(\overline{\Omega}).
\end{equation*}
By Lemma~\ref{lm34}, there exists a sequence of solutions $\{u_j\}\subset H^{s_m}(\mathbb R^n)$ satisfying
\begin{equation*}
L_{b,q}u_j = 0 \quad\text{in }\Omega,\qquad \operatorname{supp}(u_j)\subset\overline{\Omega_1},
\end{equation*}
such that $e_\Omega r_\Omega u_j\in\mathcal E^{s_m}(\overline{\Omega})$ and
\begin{equation*}
e_\Omega r_\Omega u_j \to h \quad\text{in }\mathcal E^{s_m}(\overline{\Omega})\text{ as }j\to\infty.
\end{equation*}
The higher-order approximation will follow if we show that the map
\begin{equation*}
\mathcal M : C^\infty(\overline{\Omega}) \to \mathcal E^{s_m}(\overline{\Omega}),\qquad
\mathcal M g := e_\Omega d(x)^{s_m} g
\end{equation*}
is a homeomorphism. Indeed, if $\mathcal M$ is a homeomorphism, then we can apply $\mathcal M^{-1} = d(x)^{-s_m} r_\Omega$ to obtain
\begin{equation*}
d(x)^{-s_m} r_\Omega u_j \to g \quad\text{in } C^\infty(\overline{\Omega}).
\end{equation*}
To see that $\mathcal M$ is a homeomorphism, note that it is a bijective linear map between Fr\'echet spaces. It has a closed graph: if $g_j\to g$ in $C^\infty(\overline{\Omega})$ and $\mathcal M g_j\to h$ in $\mathcal E^{s_m}(\overline{\Omega})$, then also $\mathcal M g_j\to \mathcal M g$ in $L^\infty(\Omega)$. By uniqueness of the limit, $\mathcal M g = h$ as distributional limits. Hence, by the closed graph theorem and the open mapping theorem, $\mathcal M$ is a homeomorphism. This finishes the proof.
\end{proof}

\section{Uniqueness}\label{sec4}

In this section we derive the Alessandrini-type identities relating the difference of two DtN maps to the differences of the unknown coefficients, and use them to prove uniqueness. We distinguish two cases: the homogeneous source case ($f=0$) and the non-homogeneous source case ($f\neq 0$).

Throughout this section we assume that Assumption~\ref{ass1} holds for all coefficient sets considered, so that the exterior value problems are well-posed and the DtN maps are well-defined.

\subsection{Case A: Homogeneous source ($f=0$)}

Let $(b_1,q_1)$ and $(b_2,q_2)$ be two sets of coefficients with $b_j\in W^{1-s_m,\infty}(\Omega)^n$, $q_j\in L^\infty(\Omega)$, and suppose that the corresponding DtN maps agree:
\begin{equation*}
\Lambda_{b_1,q_1} = \Lambda_{b_2,q_2}.
\end{equation*}
Here $\Lambda_{b,q}$ is the linear DtN map defined in Definition~\ref{def1}.

Let $g_1,g_2\in X=H^{s_m}(\mathbb R^n)/\widetilde H^{s_m}(\Omega)$ be two exterior data. Denote by $u_1\in H^{s_m}(\mathbb R^n)$ the solution of
\begin{equation*}
L_{b_1,q_1} u_1 = 0 \quad \text{in } \Omega,\qquad u_1-g_1\in\widetilde H^{s_m}(\Omega),
\end{equation*}
and by $u_2^*\in H^{s_m}(\mathbb R^n)$ the solution of the adjoint problem with coefficients $(b_2,q_2)$,
\begin{equation*}
L_{b_2,q_2}^* u_2^* = 0 \quad \text{in } \Omega,\qquad u_2^*-g_2\in\widetilde H^{s_m}(\Omega).
\end{equation*}
Then the following identity holds.

\begin{lemma}\label{lem:alessandrini-hom}
Under the above assumptions, we have
\begin{equation*}
\langle (\Lambda_{b_1,q_1}-\Lambda_{b_2,q_2})g_1,\,g_2\rangle
= \int_\Omega (b_1-b_2)\cdot\nabla u_1\, u_2^*\,dx
+ \int_\Omega (q_1-q_2)u_1 u_2^*\,dx.
\end{equation*}
\end{lemma}

\begin{proof}
Using the definition of the DtN map and the symmetry of the bilinear forms (see \eqref{eq28}), we have
\begin{equation*}
\langle \Lambda_{b_1,q_1} g_1,g_2\rangle
= B_{b_1,q_1}(u_1,u_2^*)
= B_{b_1,q_1}^*(u_2^*,u_1)
\end{equation*}
and
\begin{equation*}
\langle \Lambda_{b_2,q_2} g_1,g_2\rangle
= \langle g_1,\Lambda_{b_2,q_2}^* g_2\rangle
= B_{b_2,q_2}^*(u_2^*,u_1)
= B_{b_2,q_2}(u_1,u_2^*).
\end{equation*}
so
\begin{equation*}
\begin{aligned}
\langle (\Lambda_{b_1,q_1}-\Lambda_{b_2,q_2})g_1,\,g_2\rangle
&=\langle \Lambda_{b_1,q_1}g_1,\,g_2\rangle-\langle g_1,\,\Lambda^*_{b_2,q_2}g_2\rangle\\
&= B_{b_1,q_1}(u_1,u_2^*)-B_{b_2,q_2}^*(u_2^*,u_1)\\
&= \int_\Omega (b_1-b_2)\cdot\nabla u_1\, u_2^*\,dx
+ \int_\Omega (q_1-q_2)u_1 u_2^*\,dx.
\end{aligned}
\end{equation*}
\end{proof}

\begin{theorem}[Uniqueness, homogeneous case]\label{thm:unique-hom}
For \(n\ge 1\) let \(\Omega\subset\mathbb R^n\) be a bounded, open, non-empty Lipschitz domain, and let \(\frac12<s_m<1\). Let \(b_j\in W^{1-s_m,\infty}(\Omega)^n\) be two drift fields, and \(q_j\in L^\infty(\Omega)\) be potentials for \(j=1,2\). Given arbitrary open, non-empty sets \(W_1,W_2\subset\Omega_e\) suppose that the DtN maps for the equations
\begin{equation*}
(\sum_{k=1}^m (-\Delta)^{s_k}+b_j\cdot\nabla+q_j)u_j=0\quad\text{in }\Omega
\end{equation*}
satisfy
\begin{equation*}
\Lambda_{b_1,q_1}g|_{W_2}=\Lambda_{b_2,q_2}g|_{W_2},
\quad\text{for any }g\in C_c^\infty(W_1).
\end{equation*}
Then \(b_1=b_2\) and \(q_1=q_2\) in \(\Omega\).
\end{theorem}
\begin{proof}
Since
\begin{equation*}
\Lambda_{b_1,q_1}g|_{W_2}=\Lambda_{b_2,q_2}g|_{W_2}
\quad\text{for any }g\in C_c^\infty(W_1),
\end{equation*}
where \(W_1,W_2\) are arbitrary, but fixed non-empty open sets in \(\Omega_e\). Now we seek to recover the potential and the drift coefficients separately.

\emph{Step1: Recovery of the potential $q$}.
Choose an arbitrary test function $\psi_2\in C_c^\infty(\Omega)$. Then choose
\(\psi_1\in C_c^\infty(\Omega)\) such that \(\psi_1=1\) on the set
\(\operatorname{supp}(\psi_2)\Subset\Omega\). By Lemma \ref{lem:Runge}, there exist sequences of solutions
\((u_j^1)_{j=1}^\infty\) and \((u_j^{2,*})_{j=1}^\infty\) in \(H^{s_m}(\mathbb R^n)\) such that
\begin{equation}
\label{eq51}
\begin{aligned}
(\sum_{k=1}^m (-\Delta)^{s_k}+b_1\cdot\nabla+q_1)u_j^1=0,\quad
&\sum_{k=1}^m (-\Delta)^{s_k} u_j^{2,*}-\nabla\cdot(b_2u_j^{2,*})+q_2u_j^{2,*}=0
\quad\text{in }\Omega,\\
\operatorname{supp}(u_j^1)\subset\overline{\Omega_1},\qquad
&\operatorname{supp}(u_j^{2,*})\subset\overline{\Omega_2},\\
u_j^1|_\Omega=\psi_1+r_j^1,\qquad
&u_j^{2,*}|_\Omega=\psi_2+r_j^{2,*},
\end{aligned}
\end{equation}
where \(\Omega_1,\Omega_2\) are non-empty open sets in \(\mathbb R^n\) containing
\(\overline{\Omega}\), and \(r_j^1,r_j^{2,*}\to0\) strongly in
\(\widetilde H^{s_m}(\Omega)\) as \(j\to\infty\).

Substituting $u_{1,j}$ and $u_{2,j}^*$ into Lemma \ref{lem:alessandrini-hom}, we obtain for every $j$,
\begin{equation*}
\int_\Omega (b_1-b_2)\cdot\nabla u_{1,j}\, u_{2,j}^*\,dx
+ \int_\Omega (q_1-q_2)u_{1,j}u_{2,j}^*\,dx = 0.
\end{equation*}
Here
\begin{equation}
\label{eq52}
\begin{aligned}
&\left|\int_\Omega (b_1-b_2)\cdot \nabla u_j^1\,u_j^{2,*}\,dx\right|\\
&\le \left|\int_\Omega (b_1-b_2)\cdot \nabla r_j^1\,r_j^{2,*}\,dx\right|
\quad+\left|\int_\Omega (b_1-b_2)\cdot \nabla r_j^1\,\psi_2\,dx\right| \\
&\quad+\left|\int_\Omega (b_1-b_2)\cdot \nabla \psi_1\,r_j^{2,*}\,dx\right|
\quad+\left|\int_\Omega (b_1-b_2)\cdot \nabla \psi_1\,\psi_2\,dx\right|
\end{aligned}
\end{equation}
where
\begin{equation*}
\label{eq52}
\begin{aligned}
&C\|b_1-b_2\|_{W^{1-s_m,\infty}(\Omega)}
\Big(
\|r_j^1\|_{\widetilde H^{s_m}(\Omega)}
\|r_j^{2,*}\|_{\widetilde H^{s_m}(\Omega)}
+\|r_j^1\|_{\widetilde H^{s_m}(\Omega)}
\|\psi_2\|_{\widetilde H^{s_m}(\Omega)} \\
&\quad+\|\psi_1\|_{\widetilde H^{s_m}(\Omega)}
\|r_j^{2,*}\|_{\widetilde H^{s_m}(\Omega)}
\Big)
\to 0 \quad\text{as }j\to\infty.
\end{aligned}
\end{equation*}
Letting $j\to\infty$ and using the strong convergence in $\widetilde H^{s_m}$ yields
\begin{equation*}
\int_\Omega (b_1-b_2)\cdot\nabla \psi_1\, \psi_2\,dx
+ \int_\Omega (q_1-q_2)\psi_1\psi_2\,dx = 0.
\end{equation*}
Since $\psi_1\equiv 1$ on $\operatorname{supp}\psi_2$, we have $\nabla\psi_1=0$ on $\operatorname{supp}\psi_2$, so $\left|\int_\Omega (b_1-b_2)\cdot \nabla \psi_1\,\psi_2\,dx\right|=0$. Thus
\begin{equation*}
\int_\Omega (q_1-q_2)\psi_2\,dx = 0.
\end{equation*}
As $\psi_2\in C_c^\infty(\Omega)$ is arbitrary, we conclude that $q_1=q_2$.

\emph{Step2: Recovery of the drift $b$.}
Now that $q_1=q_2$, Lemma \ref{lem:alessandrini-hom} becomes
\begin{equation}
\label{eq53}
\int_\Omega (b_1-b_2)\cdot\nabla u_1\,u_2^*\,dx=0.
\end{equation}
Fix an index $i\in\{1,\dots,n\}$ and an arbitrary test function $\psi_2\in C_c^\infty(\Omega)$. Choose $\psi_{x_i}\in C_c^\infty(\Omega)$ such that $\psi_{x_i}(x)=x_i$ on $\operatorname{supp}\psi_2$. By Lemma  \ref{lem:Runge}, we can find sequences of solutions
\(u_j^1=\psi_{x_i}+r_j^1\) and \(u_j^{2,*}=\psi_2+r_j^{2,*}\),
with \(r_j^1,r_j^{2,*}\to0\) strongly in \(\widetilde H^{s_m}(\Omega)\) as \(j\to\infty\).
Plugging the functions \(\psi_{x_i}\) and \(\psi_2\) into \ref{eq53}, we obtain
\begin{equation*}
\begin{aligned}
\int_{\Omega}(b_1-b_2)_i\psi_2\,dx
&+\int_{\Omega}(b_1-b_2)\cdot\nabla\psi_{x_i}\,r_j^{2,*}\,dx\\
&+\int_{\Omega}(b_1-b_2)\cdot\nabla r_j^1\,\psi_2\,dx
+\int_{\Omega}(b_1-b_2)\cdot\nabla r_j^1\,r_j^{2,*}\,dx=0
\end{aligned}
\end{equation*}
for \(i=1,2,\ldots,n\), where \((b_1-b_2)_k\) denotes the \(i\)-th component of the vector-valued function \(b_1-b_2\). By arguing similarly as in \eqref{eq52}, in the limit \(j\to\infty\) we arrive at
\begin{equation*}
\int_\Omega (b_1-b_2)_i\psi_2\,dx=0
\quad\text{for }i=1,2,\ldots,n.
\end{equation*}
Since \(\psi_2\in C_c^\infty(\Omega)\) is arbitrary and as \(C_c^\infty(\Omega)\) is dense in \(L^2(\Omega)\), we also conclude that
\((b_1-b_2)_i=0\) for all \(i=1,2,\ldots,n\). Therefore, \(b_1=b_2\), which completes the proof.
\end{proof}

%
%

\subsection{Case B: Non-homogeneous source ($f\neq 0$)}

We now turn to the non-homogeneous case. Let $(b_1,q_1,f_1)$ and $(b_2,q_2,f_2)$ be two triples with $b_j\in W^{1-s_m,\infty}(\Omega)^n$, $q_j\in L^\infty(\Omega)$, $f_j\in L^2(\Omega)$. Suppose that the full DtN maps (which are affine in the exterior data) coincide:
\begin{equation*}
\Lambda_{b_1,q_1,f_1} = \Lambda_{b_2,q_2,f_2}.
\end{equation*}
We will show that $b_1=b_2$, $q_1=q_2$, and $f_1=f_2$.

The definition of $\Lambda_{b,q,f}$ is independent of the choice of the
representative $v_h$ of the class $h\in X$. Indeed, if $v_h,v_h'\in
H^{s_m}(\mathbb R^n)$ satisfy $v_h-v_h'\in\widetilde H^{s_m}(\Omega)$, then
by the weak formulation of the non-homogeneous problem,
\begin{equation*}
B_{b,q}(u_g,\,v_h-v_h')
= \int_\Omega f\,(v_h-v_h')\,dx,
\end{equation*}
so that
\begin{equation*}
B_{b,q}(u_g,v_h) - \int_\Omega f\,v_h\,dx
=
B_{b,q}(u_g,v_h') - \int_\Omega f\,v_h'\,dx.
\end{equation*}
Hence $\Lambda_{b,q,f}$ is well-defined. Moreover, when $f=0$, the
definition above reduces to that of the linear DtN map $\Lambda_{b,q}$ in
Definition~\ref{def1}.

First we define the DtN map for the non-homogeneous problem. For a given exterior datum $g\in X$, let $u_g\in H^{s_m}(\mathbb R^n)$ be the unique solution of
\begin{equation*}
L_{b,q,f}u_g = f \quad \text{in } \Omega,\qquad u_g-g\in\widetilde H^{s_m}(\Omega).
\end{equation*}
Then for any $h\in X$, we define
\begin{equation*}
\langle \Lambda_{b,q,f}(g),\,h\rangle
:= B_{b,q}(u_g,v_h) - \int_\Omega f\, v_h\,dx,
\end{equation*}
where $v_h$ is any representative of $h$. This definition is independent of the choice of representative and is consistent with the homogeneous case when $f=0$.

Let $g_1,g_2\in X$ be two exterior data. Let $u_1$ be the solution corresponding to $(b_1,q_1,f_1)$ with exterior data $g_1$. Let $u_2^*$ be a solution of the homogeneous adjoint equation with coefficients $(b_2,q_2)$:
\begin{equation*}
L_{b_2,q_2,0}^* u_2^* = 0 \quad \text{in } \Omega,\qquad u_2^*-g_2\in\widetilde H^{s_m}(\Omega).
\end{equation*}

\begin{lemma}\label{lem:alessandrini-nonhom}
Under the above assumptions,
\begin{equation*}
\langle (\Lambda_{b_1,q_1,f_1}-\Lambda_{b_2,q_2,f_2})g_1,\,g_2\rangle
= \int_\Omega (b_1-b_2)\cdot\nabla u_1\, u_2^*\,dx
+ \int_\Omega (q_1-q_2)u_1u_2^*\,dx
+ \int_\Omega (f_2-f_1)u_2^*\,dx.
\end{equation*}
\end{lemma}

\begin{proof}
By definition,
\begin{equation*}
\langle \Lambda_{b_1,q_1,f_1} g_1,g_2\rangle
= B_{b_1,q_1}(u_1,u_2^*) - \int_\Omega f_1 u_2^*\,dx.
\end{equation*}
For the second map, we get
\begin{equation*}
\langle \Lambda_{b_2,q_2,f_2} g_1,g_2\rangle
= \langle g_1,\Lambda_{b_2,q_2,f_2}^* g_2\rangle
= B_{b_2,q_2}^*(u_2^*,u_1) - \int_\Omega f_2 u_2^*\,dx
= B_{b_2,q_2}(u_1,u_2^*) - \int_\Omega f_2 u_2^*\,dx.
\end{equation*}
Thus
\begin{equation*}
\begin{aligned}
\langle (\Lambda_{b_1,q_1,f_1}-\Lambda_{b_2,q_2,f_2})g_1,\,g_2\rangle &=\langle \Lambda_{b_1,q_1,f_1} g_1,g_2\rangle-\langle g_1,\Lambda_{b_2,q_2,f_2}^* g_2\rangle\\
&=B_{b_1,q_1}(u_1,u_2^*) - \int_\Omega f_1 u_2^*\,dx-B_{b_2,q_2}(u_1,u_2^*) + \int_\Omega f_2 u_2^*\,dx\\
&= \int_\Omega (b_1-b_2)\cdot\nabla u_1\, u_2^*\,dx\\
&+ \int_\Omega (q_1-q_2)u_1u_2^*\,dx\\
&+ \int_\Omega (f_2-f_1)u_2^*\,dx.
\end{aligned}
\end{equation*}
\end{proof}

Now we prove uniqueness.

\begin{theorem}[Uniqueness, non-homogeneous case]\label{thm:unique-nonhom}
For \(n\ge 1\) let \(\Omega\subset\mathbb R^n\) be a bounded, open, non-empty Lipschitz domain, and let \(\frac12<s_m<1\). Let \(b_j\in W^{1-s_m,\infty}(\Omega)^n\) be two drift fields, \(q_j\in L^\infty(\Omega)\) be potentials, and $f_j\in L^2(\Omega)$ be source term for \(j=1,2\). Given arbitrary open, non-empty sets \(W_1,W_2\subset\Omega_e\) suppose that the DtN maps for the equations
\begin{equation*}
(\sum_{k=1}^m (-\Delta)^{s_k}+b_j\cdot\nabla+q_j)u_j=f_j\quad\text{in }\Omega
\end{equation*}
satisfy
\begin{equation*}
\Lambda_{b_1,q_1,f_1}g=\Lambda_{b_2,q_2,f_2}g,
\quad\text{for any }g\in C_c^\infty(W_1).
\end{equation*}
Then $b_1=b_2$, $q_1=q_2$, and $f_1=f_2$ in $\Omega$.
\end{theorem}

\begin{proof}
From Lemma~\ref{lem:alessandrini-nonhom}, since the DtN maps are equal, the left-hand side is zero for all $g_1,g_2$. Hence, for all admissible $u_1$ and $u_2^*$, we have
\begin{equation}\label{eq:nonhom-identity}
\int_\Omega (b_1-b_2)\cdot\nabla u_1\, u_2^*\,dx
+ \int_\Omega (q_1-q_2)u_1u_2^*\,dx
+ \int_\Omega (f_2-f_1)u_2^*\,dx = 0.
\end{equation}
\emph{Step 1: Recover $f$.}
Fix an arbitrary test function $\psi_2\in C_c^\infty(\Omega)$. Choose another test function $\psi_1\in C_c^\infty(\Omega)$ such that $\psi_1\equiv 0$ on $\operatorname{supp}\psi_2$.

By Lemma \ref{lem:Runge}, there exist sequences of solutions
\((u_j^1)_{j=1}^\infty\) and \((u_j^{2,*})_{j=1}^\infty\) in \(H^{s_m}(\mathbb R^n)\) such that
\begin{equation}
\label{eq521}
\begin{aligned}
(\sum_{k=1}^m (-\Delta)^{s_k}+b_1\cdot\nabla+q_1)u_j^1=f,\quad
&\sum_{k=1}^m (-\Delta)^{s_k} u_j^{2,*}-\nabla\cdot(b_2u_j^{2,*})+q_2u_j^{2,*}=0
\quad\text{in }\Omega,\\
\operatorname{supp}(u_j^1)\subset\overline{\Omega_1},\qquad
&\operatorname{supp}(u_j^{2,*})\subset\overline{\Omega_2},\\
u_j^1|_\Omega=\psi_1+r_j^1,\qquad
&u_j^{2,*}|_\Omega=\psi_2+r_j^{2,*},
\end{aligned}
\end{equation}
where \(\Omega_1,\Omega_2\) are non-empty open sets in \(\mathbb R^n\) containing
\(\overline{\Omega}\), and \(r_j^1,r_j^{2,*}\to0\) strongly in
\(\widetilde H^{s_m}(\Omega)\) as \(j\to\infty\).

For the non-homogeneous equation, we use the superposition principle: the solution $u_j^1$ corresponding to exterior data $g_j^1$ can be written as $u_j^1=P_{b_1,q_1}g_j^1+R_{b_1,q_1}f_1$, where $P_{b_1,q_1}g_j^1$ solves \eqref{eq210} and $R_{b_1,q_1}f_1$ solves \eqref{eq211}. Since $R_{b_1,q_1}f_1$ is fixed, we can use the Runge approximation for the homogeneous equation to choose $g_j^1$ such that
\begin{equation*}
P_{b_1,q_1}g_j^1 \to \psi_1 - R_{b_1,q_1}f_1 \quad \text{in } \widetilde H^{s_m}(\Omega).
\end{equation*}
Then $u_j^1\to \psi_1$ in $\widetilde H^{s_m}(\Omega)$.

Substituting these sequences into \eqref{eq:nonhom-identity} and passing to the limit gives
\begin{equation*}
\int_\Omega (b_1-b_2)\cdot\nabla \psi_1\, \psi_2\,dx
+ \int_\Omega (q_1-q_2)\psi_1\psi_2\,dx
+ \int_\Omega (f_2-f_1)\psi_2\,dx = 0.
\end{equation*}
Since $\psi_1=0$ on $\operatorname{supp}\psi_2$, the first two integrals vanish. Therefore
\begin{equation*}
\int_\Omega (f_2-f_1)\psi_2\,dx = 0.
\end{equation*}
As $\psi_2\in C_c^\infty(\Omega)$ is arbitrary, we conclude $f_1=f_2$ a.e.\ in $\Omega$.

\emph{Step 2: Recover $q$.}
Now that $f_1=f_2$, the identity \eqref{eq:nonhom-identity} simplifies to
\begin{equation*}
\int_\Omega (b_1-b_2)\cdot\nabla u_j^1\, u_j^{2,*}\,dx
+ \int_\Omega (q_1-q_2)u_j^1u_j^{2,*}\,dx = 0.
\end{equation*}
Choose an arbitrary $\psi_2\in C_c^\infty(\Omega)$ and take $\psi_1\in C_c^\infty(\Omega)$ such that $\psi_1\equiv 1$ on $\operatorname{supp}\psi_2$. As in the homogeneous case, we approximate $\psi_1$ by a sequence $u_{1,j}$ of non-homogeneous solutions and approximate $\psi_2$ by solutions $u_{2,j}^*$ of the homogeneous adjoint equation. Passing to the limit yields
\begin{equation*}
\int_\Omega (b_1-b_2)\cdot\nabla \psi_1\, \psi_2\,dx
+ \int_\Omega (q_1-q_2)\psi_1\psi_2\,dx = 0.
\end{equation*}
The first term vanishes because $\nabla\psi_1=0$ on $\operatorname{supp}\psi_2$. Hence
\begin{equation*}
\int_\Omega (q_1-q_2)\psi_2\,dx = 0,
\end{equation*}
and since $\psi_2$ is arbitrary, $q_1=q_2$ a.e.\ in $\Omega$.

\emph{Step 3: Recover $b$.}
With $q_1=q_2$ and $f_1=f_2$, identity \eqref{eq:nonhom-identity} reduces to
\begin{equation*}
\int_\Omega (b_1-b_2)\cdot\nabla u_j^1\, u_j^{2,*}\,dx = 0.
\end{equation*}
For each coordinate $k$, choose $\psi_2\in C_c^\infty(\Omega)$ and let $\psi_{x_i}\in C_c^\infty(\Omega)$ be such that $\psi_{x_i}=x_i$ on $\operatorname{supp}\psi_2$. By the Runge approximation, we obtain sequences converging to $\psi_{x_i}$ and $\psi_2$. Substituting and passing to the limit gives
\begin{equation*}
\int_\Omega (b_1-b_2)\cdot\nabla \psi_{x_i}\, \psi_2\,dx = 0.
\end{equation*}
Since $\nabla\psi_{x_i}=e_k$ on $\operatorname{supp}\psi_2$, we get
\begin{equation*}
\int_\Omega (b_1-b_2)_k \psi_2\,dx = 0,
\end{equation*}
which implies $(b_1-b_2)_k=0$ for all $k$. Thus $b_1=b_2$.

This completes the proof of the theorem.
\end{proof}

\section{Numerical method}\label{sec5}

In this section we develop a complete numerical framework for the inverse problems studied. The method consists of three main components: (i) a finite difference solver for the multi-fractional Schr\"odinger equation with drift (the forward problem), (ii) a radial basis function (RBF) surrogate model that replaces the expensive forward solver during inversion, and (iii) a hierarchical Bayesian inversion framework that simultaneously reconstructs the unknown coefficients $b$, $q$, $f$ and quantifies the reconstruction uncertainty.

\subsection{Parameterization}

We introduce the following parameterizations,

For drift coefficient $b=(b_1,\dots,b_n)$,  potential $q$, and source term $f$, we respectively parameterize them,
 \begin{equation*}
    b_j(x)=\sum_{k=1}^{N_b} \beta_{j,k}\,\psi_k(x),\ q(x)=\sum_{k=1}^{N_q} \theta_k\,\phi_k(x),\ f(x)=\sum_{k=1}^{N_f} \eta_k\,\chi_k(x),\qquad j=1,\dots,n,
    \end{equation*}
where $\psi_k, \phi_k(x)$, and $\chi_k(x)$ are basis functions with compact support in $\Omega$.
We denote the collection of all coefficients by $\boldsymbol\beta\in\mathbb R^{M_b}$, $\boldsymbol\theta\in\mathbb R^{M_q}$, $\boldsymbol\eta\in\mathbb R^{M_f}.$

The overall parameter vector is
\begin{equation*}
\boldsymbol\xi=(\boldsymbol\beta,\boldsymbol\theta,\boldsymbol\eta)\in\mathbb R^M,\qquad M=M_b+M_q+M_f.
\end{equation*}

\subsection{Finite difference discretization of the forward problem}

Let $\Omega=(-L,L)$. Following the standard truncation approach for fractional problems \cite{Li2023, Li2025}, we introduce a truncation parameter $R>L$ and consider the truncated problem
\begin{equation}
\label{eq:truncated}
\begin{cases}
\left(\sum_{k=1}^m (-\Delta)^{s_k} + b(x)\partial_x + q(x)\right)u_R(x)=f(x), & x\in(-L,L),\\
u_R(x)=g(x), & x\in(-R,R)\setminus(-L,L),\\
u_R(x)=0, & x\in\mathbb R\setminus(-R,R),
\end{cases}
\end{equation}
where $g\in C_c^\infty(W)$ is the prescribed exterior datum supported in $W\subset\Omega_e$.

Let $N$ be a positive integer, $h=2L/N$, and define the interior grid points $x_i=-L+ih$, $i=1,\dots,N-1$.
\begin{equation*}
(-\Delta)^s u_R(x_i)
= -c_{1,s}\int_0^\infty \frac{u_R(x_i-\xi)-2u_R(x_i)+u_R(x_i+\xi)}{\xi^{1+2s}}\,d\xi.
\end{equation*}

Following the splitting-parameter approach of \cite{Li2023}, we introduce $\gamma\in(2s,2]$ and rewrite the integral as
\begin{equation*}
(-\Delta)^s u_R(x_i)
= -c_{1,s}\left(\int_0^{2L} \frac{\psi_\gamma(x_i,\xi)}{\xi^{\gamma-(1+2s)}}\,d\xi
-2u_R(x_i)\int_{2L}^\infty \frac{d\xi}{\xi^{1+2s}}
+\int_{2L}^{L+R} \frac{g(x_i-\xi)+g(x_i+\xi)}{\xi^{1+2s}}\,d\xi\right),
\end{equation*}
where $\psi_\gamma(x_i,\xi)=[u_R(x_i-\xi)-2u_R(x_i)+u_R(x_i+\xi)]/\xi^\gamma$.

The resulting finite difference approximation yields the linear system
\begin{equation}
\label{eq:fd_system_1d}
\left(\mathbf L^s+\mathbf B+\mathbf Q\right)\mathbf U_R=\mathbf b_R,
\end{equation}
where:
\begin{itemize}
    \item $\mathbf L^s\in\mathbb R^{(N-1)\times(N-1)}$ is the symmetric Toeplitz matrix discretizing $\sum_{k=1}^m (-\Delta)^{s_k}$;
    \item $\mathbf B\in\mathbb R^{(N-1)\times(N-1)}$ discretizes the drift term $b(x)\partial_x$ ;
    \item $\mathbf Q=\operatorname{diag}(q(x_1),\dots,q(x_{N-1}))$;
    \item $\mathbf U_R=(u_R(x_1),\dots,u_R(x_{N-1}))^\top$;
    \item $\mathbf b_R$ collects contributions from the exterior datum $g$ and the source $f$.
\end{itemize}

The measurement data correspond to the operator $\mathcal O(u)=\sum_{k=1}^m (-\Delta)^{s_k}u$ restricted to the exterior observation set $W_2\subset\Omega_e$. For an observation point $x_\ell\in W_2$, we have
\begin{equation*}
\mathcal O(u)(x_\ell)=\sum_{k=1}^m (-\Delta)^{s_k}u_R(x_\ell)
=\sum_{j\in\mathcal I} w_{\ell,j} U_{R,j}+\sum_{k\in\mathcal E} w_{\ell,k} g_k,
\end{equation*}
where $\mathcal I$ and $\mathcal E$ denote the index sets of interior and exterior grid points, respectively, and $w_{\ell,j}$ are the finite difference weights. In vector form,
\begin{equation}
\label{eq:obs_operator}
\mathbf h=\mathbf E_{W_2}\left(\mathbf L^s\mathbf U_R+\mathbf b_R\right),
\end{equation}
where $\mathbf E_{W_2}$ restricts to the observation points and $\mathbf b_{\text{obs}}$ accounts for the known exterior contribution.

\subsection{RBF surrogate model}

We generate $N_{\mathrm{t}}$ training points $\{\boldsymbol\xi^{(m)}\}_{m=1}^{N_{\mathrm{t}}}$ over the parameter space. For each sample, we solve the forward problem \eqref{eq:fd_system_1d} and compute the observation vector $\mathbf h^{(m)}\in\mathbb R^{N_{\mathrm{o}}}$ via \eqref{eq:obs_operator},
\begin{equation*}
\mathbf H=[\mathbf h^{(1)},\dots,\mathbf h^{(N_{\mathrm{t}})}]\in\mathbb R^{N_{\mathrm{o}}\times N_{\mathrm{t}}}.
\end{equation*}

We use the inverse multiquadric  RBF \cite{Marko-2014}.
\begin{equation*}
\varphi_m(\boldsymbol\xi)
=\frac{1}{\sqrt{\|\boldsymbol\xi-\boldsymbol\xi^{(m)}\|_2^2+\sigma^2}},
\end{equation*}
where
\begin{equation*}
\sigma=\frac{1}{N_t(N_t-1)}
\sum_{i=1}^{N_t}\sum_{j\ne i}
\|\boldsymbol\xi^{(i)}-\boldsymbol\xi^{(j)}\|_2.
\end{equation*}
Let
\begin{equation*}
\boldsymbol\varphi(\boldsymbol\xi)
=[\varphi_1(\boldsymbol\xi),\dots,\varphi_{N_t}(\boldsymbol\xi)]^T,
\end{equation*}
and let
\begin{equation*}
\mathbf F_{i,n}=\varphi_{n}(\boldsymbol\xi^{(i)}).
\end{equation*}
The coefficient matrix $\mathbf C\in\mathbb R^{S\times N_t}$ is determined by
\begin{equation*}
\mathbf H=\mathbf C\mathbf F.
\end{equation*}
When $\mathbf F$ is well conditioned,
\begin{equation*}
\mathbf C=\mathbf H\mathbf F^{-1}.
\end{equation*}
Otherwise, we use the regularized solution
\begin{equation*}
\mathbf C
=\mathbf H\mathbf F^T
\left(\mathbf F\mathbf F^T+\lambda_{\mathrm{RBF}}\mathbf I\right)^{-1},
\end{equation*}
with a small $\lambda_{\mathrm{RBF}}>0$.

The surrogate prediction takes the form
\begin{equation}
\label{eq:surrogate}
\mathbf h_p(\boldsymbol\xi)
=\mathbf C\boldsymbol\varphi(\boldsymbol\xi)
=\sum_{n=1}^{N_t}\mathbf c_n\varphi_n(\boldsymbol\xi),
\end{equation}
where $\mathbf c_n$ is the $n$-th column of $\mathbf C$.

\subsection{Bayesian inversion }
There are many methods for reconstructing inverse problems, such as Tikhonov regularization, Landweber iteration, etc. In contrast, the Bayesian inverse problem method treats all variables as random variables and encodes prior knowledge through probability distributions\cite{Kaipio-2005,Stuart-2010}. The solution is the posterior distribution. It not only for quantitative parameter uncertainty provides a coherent framework, also to the natural adaptation of various data types and complex priori information model. For more articles on Bayesian methods, please refer to \cite{Giordano-2026, Marko-2014, Wang+Bardsley-2017,Lassas+Saksman-2009,Kolehmainen-2012}.

In our article, based on the prior assumption of the continuity of the potential function, we introduced the ARD prior and the TV prior separately into the hierarchical Bayesian framework. To avoid the high cost of repeatedly solving fractional partial differential equations during the inference process, we pre-built a forward mapping radial basis function substitution model and incorporated the approximation error of this model into the likelihood function. Subsequently, the Laplace approximation of the posterior distribution provided pointwise confidence intervals for the reconstructed potential function.

We now formulate the Bayesian inverse problem for reconstructing the parameter vector
\begin{equation*}
\boldsymbol\xi=(\boldsymbol\beta,\boldsymbol\theta,\boldsymbol\eta)\in\mathbb R^M
\end{equation*}
from noisy observations
\begin{equation*}
\mathbf h^\delta=\mathbf F(\boldsymbol\xi)+\boldsymbol\varepsilon,\qquad \boldsymbol\varepsilon\sim\mathcal N(\mathbf 0,\delta^2\mathbf I).
\end{equation*}
Combining measurement noise and surrogate error, the total observation covariance is
\begin{equation*}
\boldsymbol\Gamma=\delta^2\mathbf I.
\end{equation*}
The likelihood function is
\begin{equation}
\label{eq:likelihood-global}
p(\mathbf h^\delta\mid\boldsymbol\xi)\propto
\exp\left(-\frac12\bigl(\mathbf h_p(\boldsymbol\xi)-\mathbf h^\delta\bigr)^\top
\boldsymbol\Gamma^{-1}
\bigl(\mathbf h_p(\boldsymbol\xi)-\mathbf h^\delta\bigr)\right).
\end{equation}
If the unknown function is smooth, we use the ARD Gaussian prior,
\begin{equation}
\label{eq:prior-global-ard}
p(\boldsymbol\xi\mid\boldsymbol\alpha)
=
\mathcal N\bigl(\boldsymbol\xi\mid\mathbf 0,
\operatorname{diag}(\boldsymbol\alpha^{-1})\bigr),
\end{equation}
where
\begin{equation*}
\boldsymbol\alpha=(\alpha_1,\dots,\alpha_M),
\qquad
\alpha_j\sim\mathrm{Gamma}(a_0,b_0).
\end{equation*}
Equivalently,
\begin{equation*}
p(\boldsymbol\xi\mid\boldsymbol\alpha)
\propto
\exp\left(
-\frac12\sum_{j=1}^{M}\alpha_j\xi_j^2
\right).
\end{equation*}
When the unknown function is discontinuous, we use theTV-GSM prior \cite{Marko-2014}. Let
\begin{equation*}
L=
\begin{pmatrix}
L_b & 0 & 0\\
0 & L_q & 0\\
0 & 0 & L_f
\end{pmatrix}
\in\mathbb R^{K\times M},
\end{equation*}
be the block first-difference matrix, where \(L_b\in\mathbb R^{K_b\times M_b}\), \(L_q\in\mathbb R^{K_q\times M_q}\), and \(L_f\in\mathbb R^{K_f\times M_f}\) are the first-difference matrices associated with the basis functions of \(b\), \(q\), and \(f\), respectively, and \(K=K_b+K_q+K_f\). Let the vector
\begin{equation*}
\mathbf z=L\boldsymbol\xi\in\mathbb R^K
\end{equation*}
be a discontinuous one. The TV-GSM prior takes the form
\begin{equation}
\label{eq:prior-global-tv}
p(\boldsymbol\xi\mid\mathbf r,\lambda_1)
\propto
\exp\left(
-\sum_{i=1}^K\frac{z_i^2}{2r_i}
-\frac{\lambda_1}{2}\sum_{i=1}^K r_i
\right),
\end{equation}
with latent scale variables \(r_i>0\).

We place Gamma hyperpriors on the ARD precisions,
\begin{equation*}
\alpha_j\sim\mathrm{Gamma}(a_0,b_0),\qquad j=1,\dots,M,
\end{equation*}
and on the TV regularization parameter,
\begin{equation*}
\lambda_1\sim\mathrm{Gamma}(a_\lambda,b_\lambda).
\end{equation*}
The latent variables \(r_i>0\) are treated as unknown hyperparameters and updated in the IAS algorithm.

Let
\begin{equation*}
\boldsymbol\Theta=
\begin{cases}
\boldsymbol\alpha, & \text{in smooth},\\
(\mathbf r,\lambda_1), & \text{in discontinuous}.
\end{cases}
\end{equation*}
By Bayes' theorem, the joint posterior density of the parameters and hyperparameters is
\begin{equation}
\label{eq:posterior-global}
p(\boldsymbol\xi,\boldsymbol\Theta\mid\mathbf h^\delta)
\propto
p(\mathbf h^\delta\mid\boldsymbol\xi)\,
p(\boldsymbol\xi\mid\boldsymbol\Theta)\,
p(\boldsymbol\Theta),
\end{equation}
where \(p(\mathbf h^\delta\mid\boldsymbol\xi)\) is the likelihood \eqref{eq:likelihood-global}, \(p(\boldsymbol\xi\mid\boldsymbol\Theta)\) is either the ARD prior \eqref{eq:prior-global-ard} or the TV-GSM prior \eqref{eq:prior-global-tv}, \(p(\boldsymbol\Theta)\) is the product of the Gamma hyperpriors.

Thus, the negative log-posterior is
\begin{equation}
\label{eq:neg-log-posterior-global}
\mathcal J(\boldsymbol\xi,\boldsymbol\Theta)
=
\Phi_{1}(\boldsymbol\xi)
+
\Phi_{2}(\boldsymbol\xi,\boldsymbol\Theta)
+
\Phi_{3}(\boldsymbol\Theta),
\end{equation}
\begin{equation*}
\Phi_{1}(\boldsymbol\xi)
=
\frac12
\bigl(\mathbf h_p(\boldsymbol\xi)-\mathbf h^\delta\bigr)^\top
\boldsymbol\Gamma^{-1}
\bigl(\mathbf h_p(\boldsymbol\xi)-\mathbf h^\delta\bigr).
\end{equation*}
\begin{equation*}
\Phi_{2}(\boldsymbol\xi,\boldsymbol\alpha)
=
\frac12\sum_{j=1}^M\alpha_j\xi_j^2
+\sum_{j=1}^M\bigl[b_0\alpha_j-(a_0-\tfrac12)\log\alpha_j\bigr],
\end{equation*}
\begin{equation*}
\Phi_{3}(\boldsymbol\xi,\mathbf r,\lambda_1)
=
\sum_{i=1}^K\frac{z_i^2}{2r_i}
+\frac{\lambda_1}{2}\sum_{i=1}^K r_i
-\frac12\sum_{i=1}^K\log r_i,
\qquad z_i=(L\boldsymbol\xi)_i.
\end{equation*}

We compute the maximum a posteriori (MAP) estimate using the iterative alternating sequential (IAS) algorithm \cite{Calvetti-2008, Marko-2014},
\begin{enumerate}
    \item Update \(\boldsymbol\xi\).
    Fix all hyperparameters and minimize \(\mathcal J\) with respect to \(\boldsymbol\xi\) using a Gauss--Newton method. The Jacobian \(\partial\mathbf h_p/\partial\boldsymbol\xi\) is computed analytically from the RBF surrogate,
    \begin{equation*}
    \frac{\partial\mathbf h_p}{\partial\xi_j}
    =\mathbf C\,\frac{\partial\boldsymbol\varphi}{\partial\xi_j},
    \end{equation*}
    with, for the inverse multiquadric kernel,
    \begin{equation*}
    \frac{\partial\varphi_m}{\partial\xi_j}
    =-\frac{\xi_j-\xi_j^{(m)}}{\sigma^2\bigl(r_m^2+1\bigr)^{3/2}},
    \qquad
    r_m^2=\frac{\|\boldsymbol\xi-\boldsymbol\xi^{(m)}\|_2^2}{\sigma^2}.
    \end{equation*}
    The Gauss--Newton step solves
    \begin{equation*}
    \Bigl(\mathbf J^\top\boldsymbol\Gamma^{-1}\mathbf J
    +\boldsymbol\Lambda\Bigr)\Delta\boldsymbol\xi
    =-\nabla_{\boldsymbol\xi}\mathcal J,
    \end{equation*}
    where \(\mathbf J=\partial\mathbf h_p/\partial\boldsymbol\xi\) and \(\boldsymbol\Lambda\) is the Hessian of the prior term (diagonal for ARD, tridiagonal for TV-GSM via \(L^\top\operatorname{diag}(\mathbf r^{-1})L\)).

    \item Update latent variables \(\mathbf r\) (discontinuous):
\begin{equation*}
r_i=\frac{\sqrt{1+4\lambda_1 z_i^2}-1}{2\lambda_1}.
\end{equation*}

    \item Update ARD precisions (smooth):
\begin{equation*}
\alpha_j=\frac{a_0-\frac12}{b_0+\frac12\xi_j^2},
\qquad j=1,\dots,M.
\end{equation*}

    \item Update TV regularization parameters (discontinuous):
    \begin{equation*}
    \lambda_1=\frac{a_{\lambda_1}+K-2}{b_{\lambda_1}+\frac12\sum_i r_i},
    \end{equation*}
    where \(K\) is the number of jumps.
\end{enumerate}

The iteration is terminated when the relative change in \(\mathcal J\) falls below a prescribed tolerance or when a maximum number of sweeps is reached.

\subsection{Uncertainty quantification}

We approximate the posterior distribution by a Gaussian using the Laplace approximation. The posterior covariance is $\boldsymbol\Sigma_{p}=\mathbf H_{p}^{-1}$, where
\begin{equation*}
\mathbf H_{p}=\nabla^2_{\boldsymbol\xi}\mathcal J(\boldsymbol\xi^*)
\approx \mathbf J(\boldsymbol\xi^*)^\top\boldsymbol\Gamma^{-1}\mathbf J(\boldsymbol\xi^*)+\mathbf H_{q}
\end{equation*}
is the Hessian of the negative log-posterior, $\mathbf J=\partial\mathbf h_p/\partial\boldsymbol\xi$ is the Jacobian, and $\mathbf H_{q}$ is the Hessian of the prior contribution.

For any quantity $g(\boldsymbol\xi)$, the posterior variance is approximated by the delta method:
\begin{equation*}
\operatorname{Var}[g(\boldsymbol\xi)\mid\mathbf h^\delta]
\approx \nabla_{\boldsymbol\xi}g(\boldsymbol\xi^*)^\top\boldsymbol\Sigma_{p}\nabla_{\boldsymbol\xi}g(\boldsymbol\xi^*).
\end{equation*}
The $95\%$ pointwise credible interval is therefore
\begin{equation*}
g(\boldsymbol\xi^*)\pm 1.96\sqrt{\operatorname{Var}[g(\boldsymbol\xi)\mid\mathbf h^\delta]}.
\end{equation*}
The complete numerical procedure is summarized as follows,

\begin{algorithm}[H]
\caption{Bayesian inversion for multi-fractional Calder\'on problem with drift}
\begin{algorithmic}[1]
\STATE \quad Generate $N_{\mathrm{t}}$ training parameters $\{\boldsymbol\xi^{(m)}\}$.
\STATE \quad For each $\boldsymbol\xi^{(m)}$, solve \eqref{eq:fd_system_1d} using the finite difference scheme.
\STATE \quad Compute observations $\mathbf h^{(m)}$.
\STATE \quad Build RBF surrogate \eqref{eq:surrogate}.
\STATE \quad Initialize $\boldsymbol\xi^{(0)}$, hyperparameters.
\STATE \quad \textbf{Repeat} until convergence:
\STATE \qquad Update $\boldsymbol\xi$ by Gauss-Newton minimization.
\STATE \qquad Update latent variables and hyperparameters by closed-form formulas.
\STATE \quad Compute posterior covariance $\boldsymbol\Sigma_{p}=\mathbf H_{p}^{-1}$.
\STATE \quad Compute pointwise credible intervals via the delta method.
\STATE \textbf{return} MAP estimate $\boldsymbol\xi^*$ and credible intervals.
\end{algorithmic}
\end{algorithm}

\section{Numerical examples}\label{sec6}

In this section, we present four numerical examples to illustrate the proposed inversion framework. The first two examples correspond to the homogeneous case $f=0$, where we simultaneously recover the drift coefficient $b$ and the potential $q$. The last two examples deal with the non-homogeneous case $f\neq 0$, in which all three unknowns $b$, $q$, and $f$ are reconstructed. In all examples, the forward problems are solved by the finite difference scheme described in Section~\ref{sec5}, and the inversions are carried out using the RBF surrogate model combined with hierarchical Bayesian inference.

Throughout this section, we use the following notation. The spatial domain is denoted by $\Omega$, and the exterior measurement region is $W_2\subset\Omega_e$. The exterior data are taken from the set $W_1\subset\Omega_e$, which we choose equal to $W_2$. The observation operator is
\begin{equation*}
\mathcal{O}(u)=\sum_{k=1}^{m}(-\Delta)^{s_k}u\Big|_{W_2}.
\end{equation*}
All measurements are corrupted by additive Gaussian noise with standard deviation $\delta$.

\textbf{Example 1:} one-dimensional smooth coefficients, $f=0$.
Let $\Omega=(-1,1)$ and choose two fractional orders
\begin{equation*}
s_1=0.3,\qquad s_2=0.7.
\end{equation*}
The drift coefficient is
\begin{equation*}
b(x)=0.4\sin(\pi x),
\end{equation*}
and the potential is
\begin{equation*}
q(x)=\sin(2\pi x)+0.5\cos(\pi x).
\end{equation*}
The source term, $f(x)\equiv 0$.

The exterior data set is
\begin{equation*}
W_1=(-3,-1-\epsilon)\cup(1+\epsilon,3),\qquad \epsilon=0.05,
\end{equation*}
and we use $J=6$ exterior data functions
\begin{equation*}
g_j(x)=\exp\bigl(-30(x-x_j)^2\bigr),\qquad j=1,\dots,6,
\end{equation*}
where the points $x_j$ are uniformly distributed in $W_1$. The observation set is $W_2=W_1$, and the measurement operator is
\begin{equation*}
\mathcal{O}(u)=\left((-\Delta)^{0.3}u+(-\Delta)^{0.7}u\right)\Big|_{W_2}.
\end{equation*}

For each exterior datum $g_j$, we solve the homogeneous problem
\begin{equation*}
\begin{cases}
\bigl((-\Delta)^{0.3}+(-\Delta)^{0.7}+b(x)\partial_x+q(x)\bigr)u_j=0, & x\in(-1,1),\\
u_j=g_j, & x\in W_1,
\end{cases}
\end{equation*}
and collect the measurements $\mathcal{O}(u_j)$. In Figure \ref{fig1}, the three figures (a), (b), and (c) respectively show the inversion images of the convective term $b$ at error levels $\delta=10^{-3}, 5\times10^{-4}, 10^{-4}$, while figures (d), (e), and (f) respectively display the inversion images of the dynamic term $q$ at the same error levels.
\begin{figure}[htbp]
	\centering
    \captionsetup{font={small}}
	{
		\begin{minipage}{0.3\linewidth}
			\centering
			\includegraphics[height=0.7\columnwidth,width=0.9\linewidth]{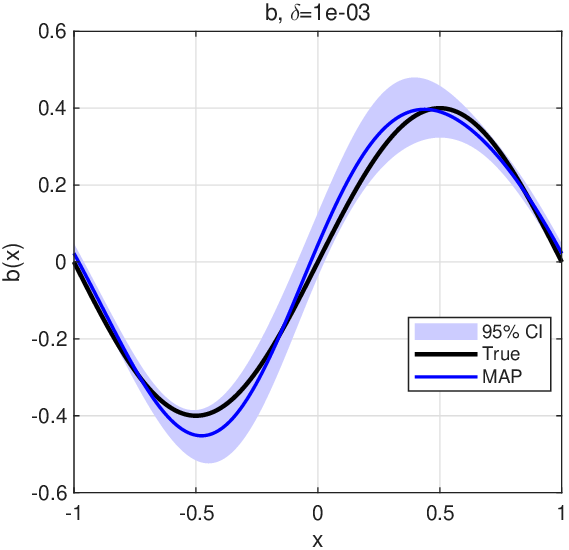}\\
			(a)
		\end{minipage}
		\begin{minipage}{0.3\linewidth}
			\centering
			\includegraphics[height=0.7\columnwidth,width=0.9\linewidth]{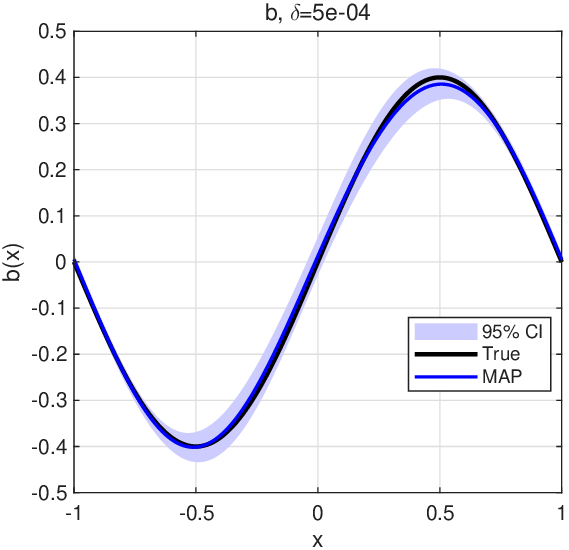}\\
			(b)
		\end{minipage}
		\begin{minipage}{0.3\linewidth}
			\centering
			\includegraphics[height=0.7\columnwidth,width=0.9\linewidth]{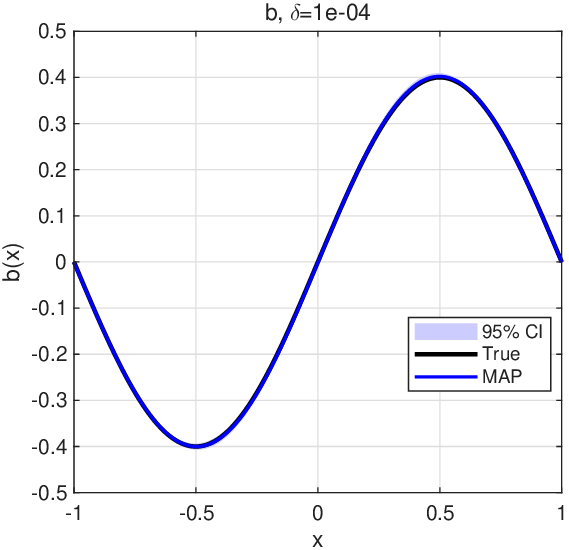}\\
			(c)
		\end{minipage}
	}
\\
	{
		\begin{minipage}{0.3\linewidth}
			\centering
			\includegraphics[height=0.7\columnwidth,width=0.9\linewidth]{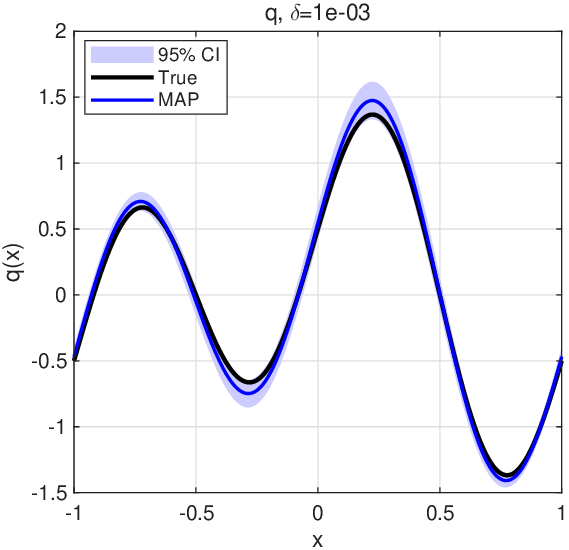}\\
			(d)
		\end{minipage}
		\begin{minipage}{0.3\linewidth}
			\centering
			\includegraphics[height=0.7\columnwidth,width=0.9\linewidth]{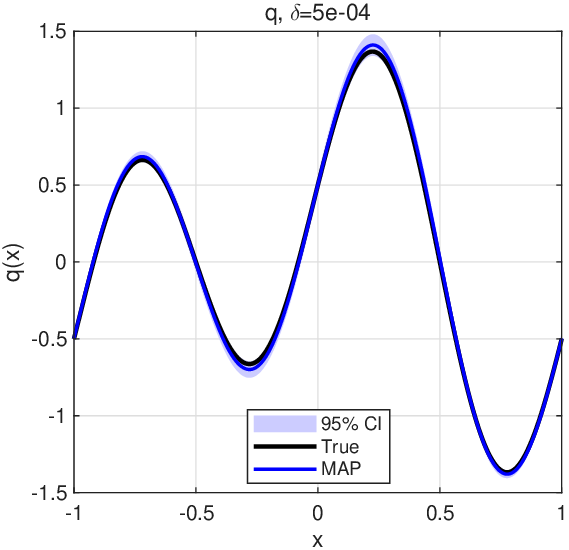}\\
			(e)
		\end{minipage}
		\begin{minipage}{0.3\linewidth}
			\centering
			\includegraphics[height=0.7\columnwidth,width=0.9\linewidth]{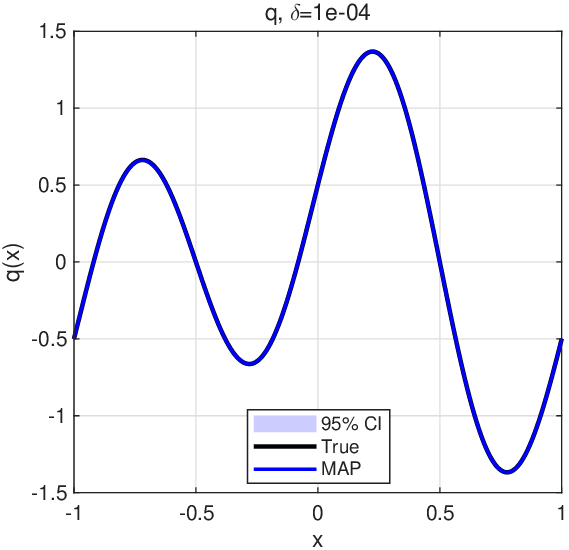}\\
			(f)
		\end{minipage}
	}
\caption{The numerical results for Example 1.}
	\label{fig1}
\end{figure}

\textbf{Example 2:} one-dimensional discontinuous potential, $f=0$.
Let $\Omega=(-1,1)$ and choose
\begin{equation*}
s_1=0.4,\qquad s_2=0.8.
\end{equation*}
The  drift is,
\begin{equation*}
b(x)=
\begin{cases}
1+x, & x\le 0,\\
1-x, & \text{otherwise}.
\end{cases}
\end{equation*}
and the potential is
\begin{equation*}
q(x)=
\begin{cases}
1, & |x|\le 0.5,\\
0, & \text{otherwise}.
\end{cases}
\end{equation*}
The source term $f(x)\equiv 0$.

As in Example~1, we take
\begin{equation*}
W_1=(-3,-1-\epsilon)\cup(1+\epsilon,3),\qquad \epsilon=0.05,
\end{equation*}
and use $J=8$ exterior data functions
\begin{equation*}
g_j(x)=\exp\bigl(-50(x-x_j)^2\bigr),\qquad j=1,\dots,8,
\end{equation*}
with $x_j\in W_1$ uniformly distributed. The observation set is $W_2=W_1$, and
\begin{equation*}
\mathcal{O}(u)=\left((-\Delta)^{0.4}u+(-\Delta)^{0.8}u\right)\Big|_{W_2}.
\end{equation*}
For each $g_j$, we solve
\begin{equation*}
\begin{cases}
\bigl((-\Delta)^{0.4}+(-\Delta)^{0.8}+b(x)\partial_x+q(x)\bigr)u_j=0, & x\in(-1,1),\\
u_j=g_j, & x\in W_1,
\end{cases}
\end{equation*}
and measure $\mathcal{O}(u_j)$. In Figure \ref{fig2}, the four figures (a), (b), (c), (d)respectively show the inversion images of the convective term $b$ and term $q$ at error levels $\delta=10^{-5},\ 10^{-6}$.
\begin{figure}[htbp]
	\centering
    \captionsetup{font={small}}
	{
		\begin{minipage}{0.45\linewidth}
			\centering
			\includegraphics[height=0.7\columnwidth,width=0.9\linewidth]{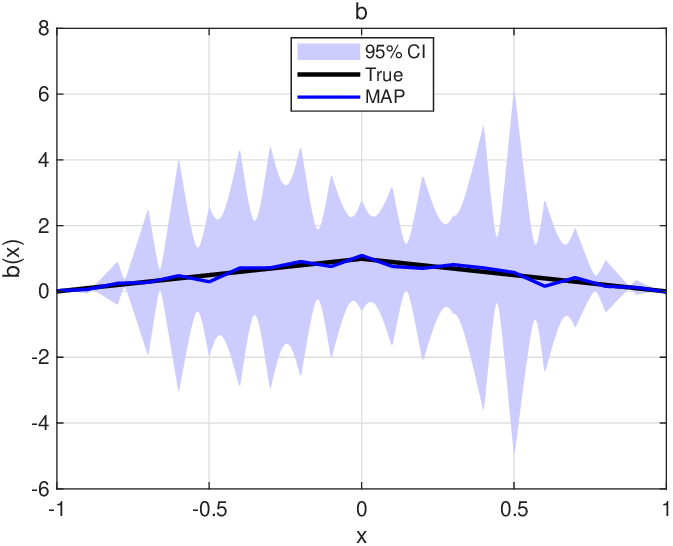}\\
		(a)
		\end{minipage}
		\begin{minipage}{0.45\linewidth}
			\centering
			\includegraphics[height=0.7\columnwidth,width=0.9\linewidth]{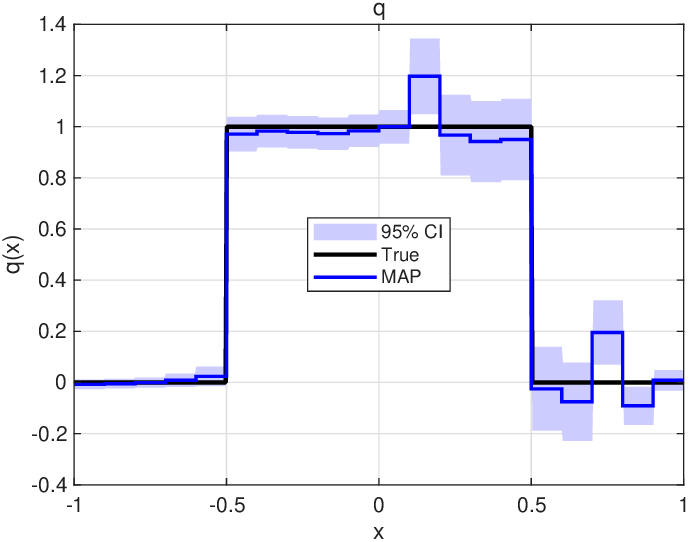}\\
		(b)
		\end{minipage}
	}
\\
	{
		\begin{minipage}{0.45\linewidth}
			\centering
			\includegraphics[height=0.7\columnwidth,width=0.9\linewidth]{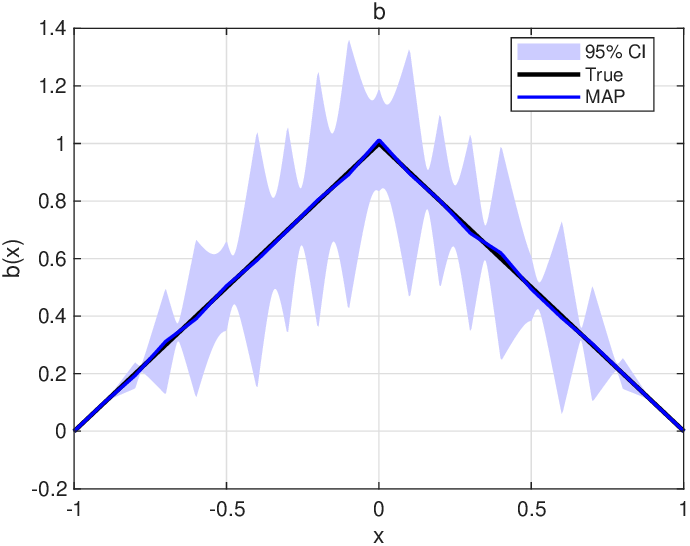}\\
		(c)
		\end{minipage}
		\begin{minipage}{0.45\linewidth}
			\centering
			\includegraphics[height=0.7\columnwidth,width=0.9\linewidth]{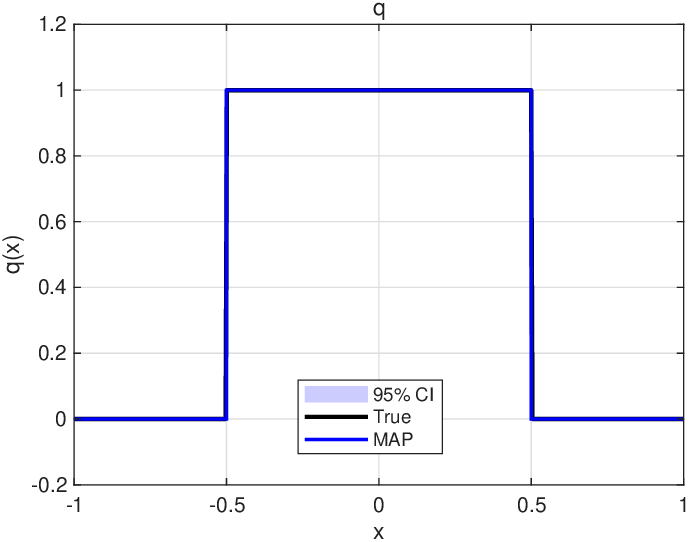}\\
		(d)
		\end{minipage}
	}
\caption{The numerical results for Example 2.}
	\label{fig2}
\end{figure}

\textbf{Example 3:} one-dimensional non-homogeneous source, simultaneous recovery of $b,q,f$.
Let $\Omega=(-1,1)$ and choose
\begin{equation*}
s_1=0.3,\qquad s_2=0.65.
\end{equation*}
The coefficients are
\begin{equation*}
b(x)=0.3\cos(2\pi x),\qquad
q(x)=0.5\sin(\pi x)+0.2,
\end{equation*}
and the source term is
\begin{equation*}
f(x)=2\,\mathbf{1}_{(-0.4,0.4)}(x)\,e^{-20x^2}.
\end{equation*}
The exterior data set is the same as before:
\begin{equation*}
W_1=(-3,-1-\epsilon)\cup(1+\epsilon,3),\qquad \epsilon=0.05,
\end{equation*}
and we use $J=8$ exterior data functions
\begin{equation*}
g_j(x)=\exp\bigl(-40(x-x_j)^2\bigr),\qquad j=1,\dots,8.
\end{equation*}
The observation set is $W_2=W_1$, and the measurement operator is
\begin{equation*}
\mathcal{O}(u)=\left((-\Delta)^{0.3}u+(-\Delta)^{0.65}u\right)\Big|_{W_2}.
\end{equation*}
For each $g_j$, we solve the non-homogeneous problem
\begin{equation*}
\begin{cases}
\bigl((-\Delta)^{0.3}+(-\Delta)^{0.65}+b(x)\partial_x+q(x)\bigr)u_j=f(x), & x\in(-1,1),\\
u_j=g_j, & x\in W_1,
\end{cases}
\end{equation*}
and collect the measurements $\mathcal{O}(u_j)$. In Figure \ref{fig3}, the three figures (a), (b), and (c) respectively show the inversion images of the convective term $b$, the potential term $q$, and the source term $f$ at the error level of  $\delta=10^{-3}$. While figures (d), (e), and (f) respectively present the inversion images at the error level of $\delta=10^{-4}$.
\begin{figure}[htbp]
	\centering
    \captionsetup{font={small}}
    	{
		\begin{minipage}{0.3\linewidth}
			\centering
			\includegraphics[height=0.7\columnwidth,width=0.9\linewidth]{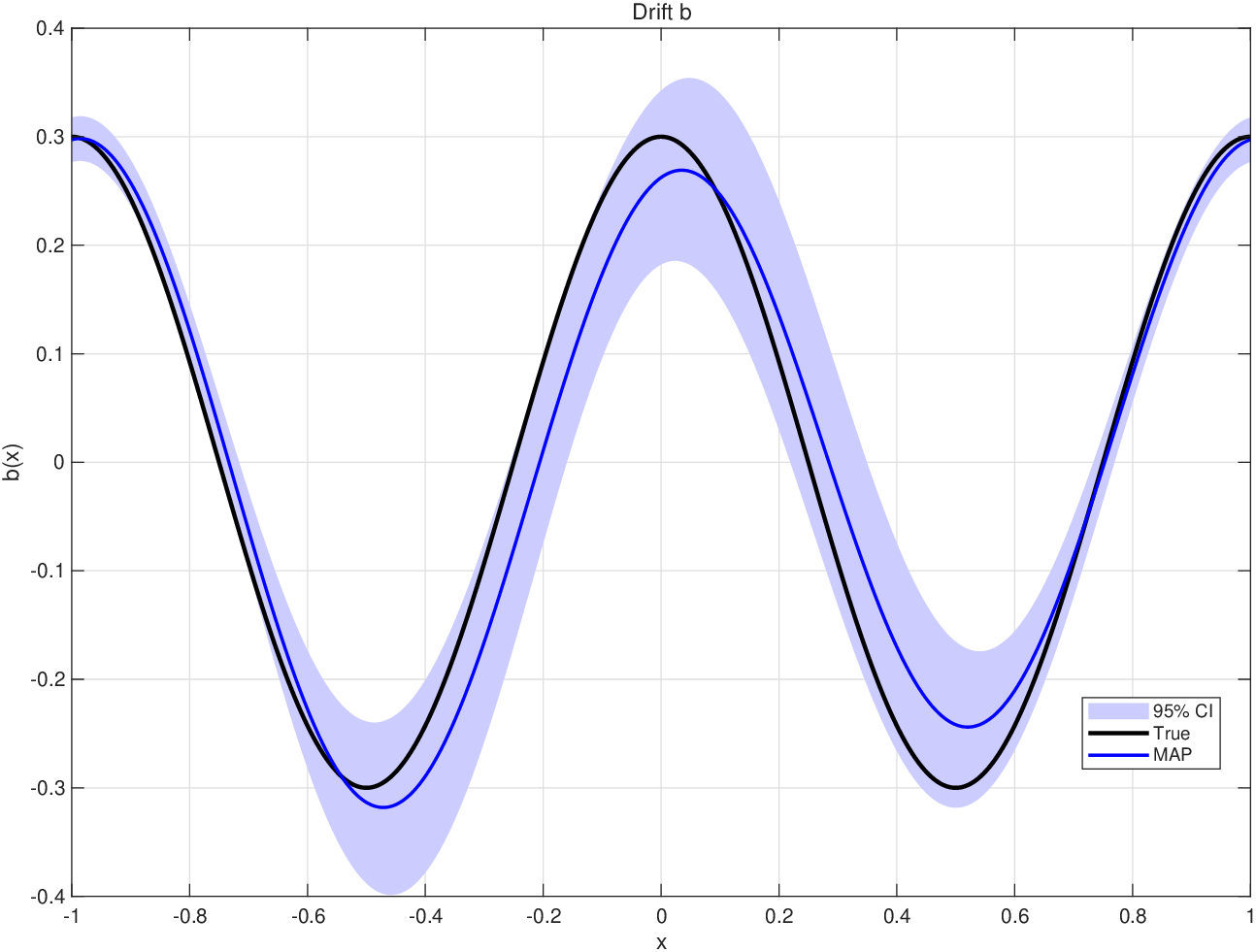}\\
		(a)
		\end{minipage}
		\begin{minipage}{0.3\linewidth}
			\centering
			\includegraphics[height=0.7\columnwidth,width=0.9\linewidth]{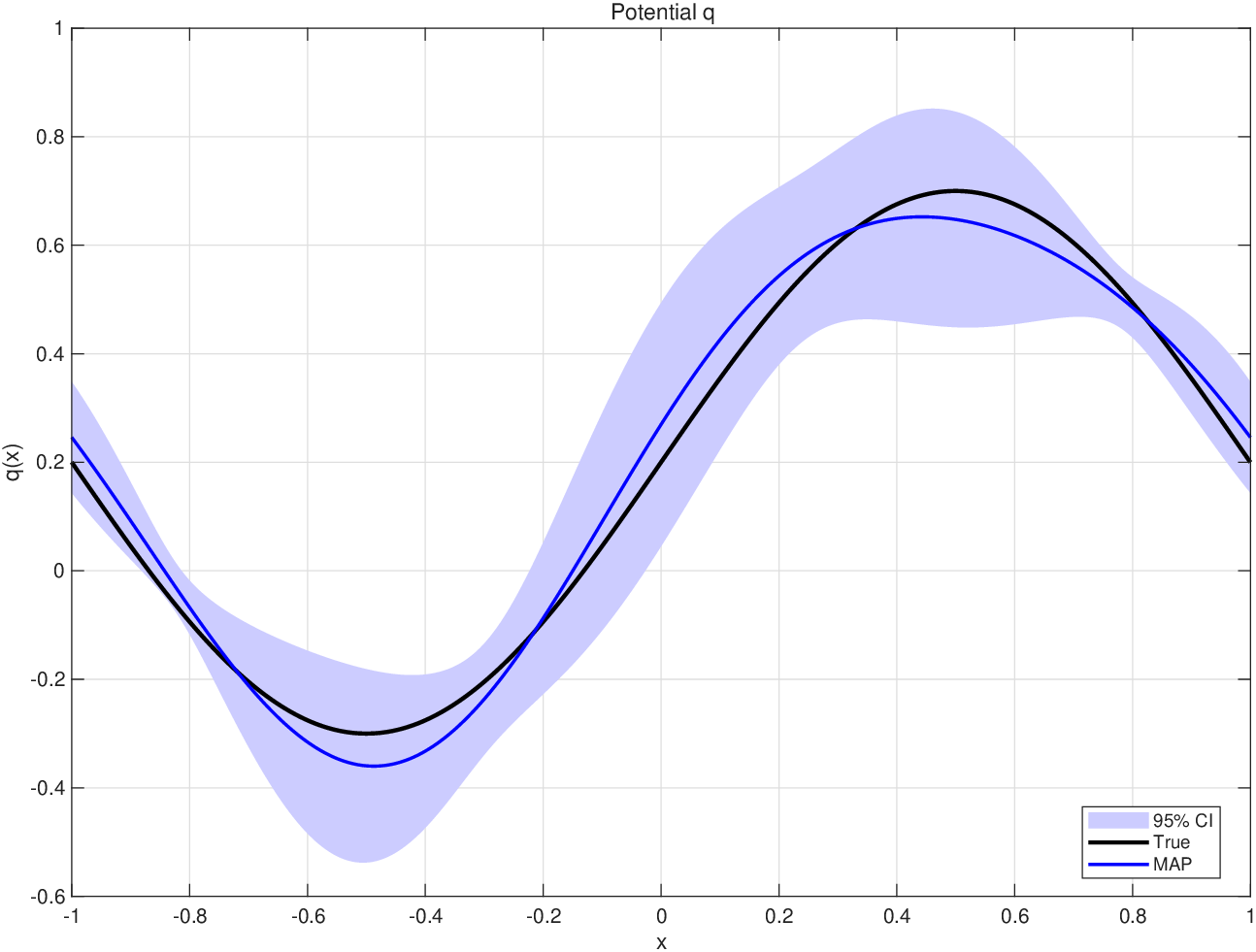}\\
		(b)
		\end{minipage}
		\begin{minipage}{0.3\linewidth}
			\centering
			\includegraphics[height=0.7\columnwidth,width=0.9\linewidth]{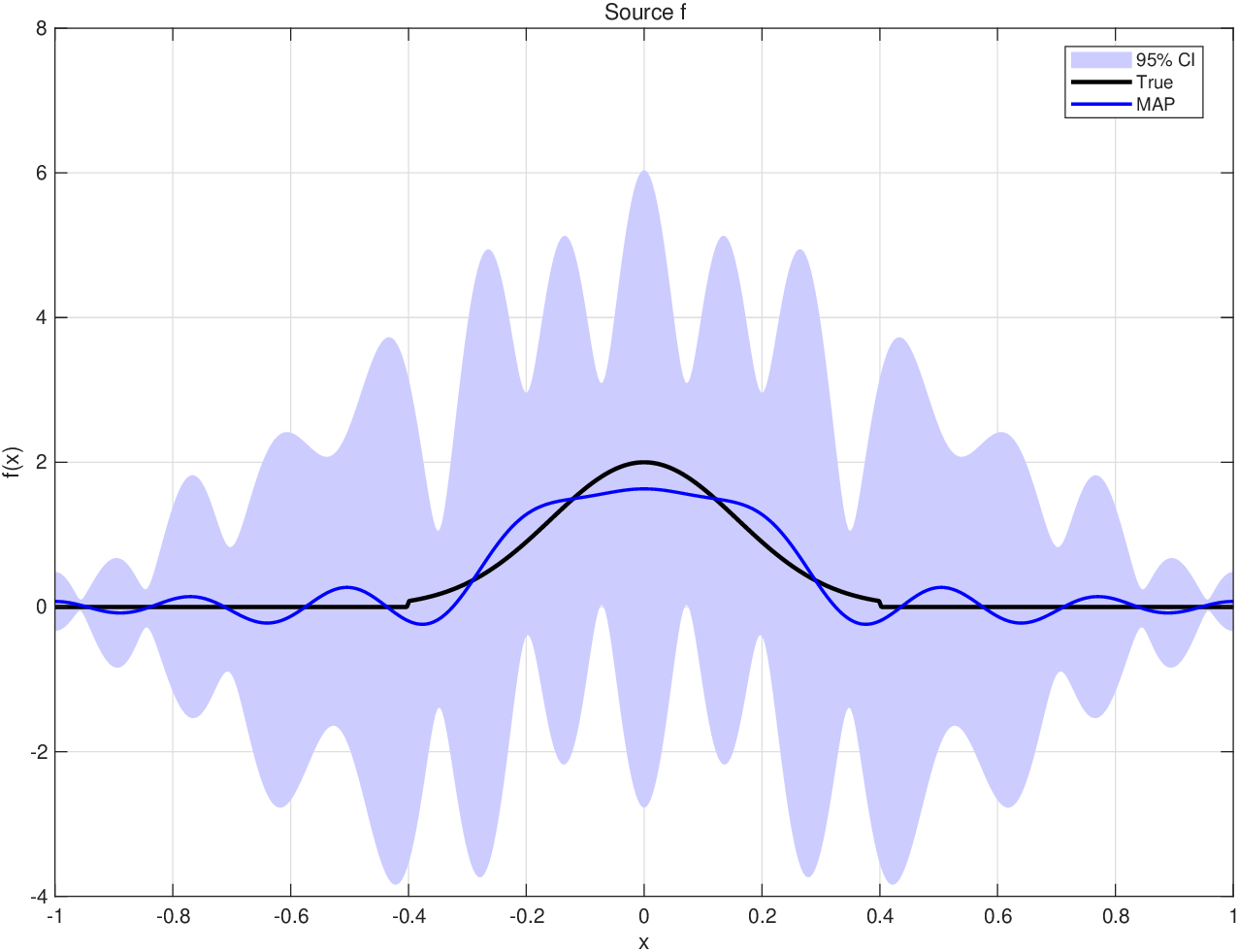}\\
		(c)
		\end{minipage}
}
\\
	{
		\begin{minipage}{0.3\linewidth}
			\centering
			\includegraphics[height=0.7\columnwidth,width=0.9\linewidth]{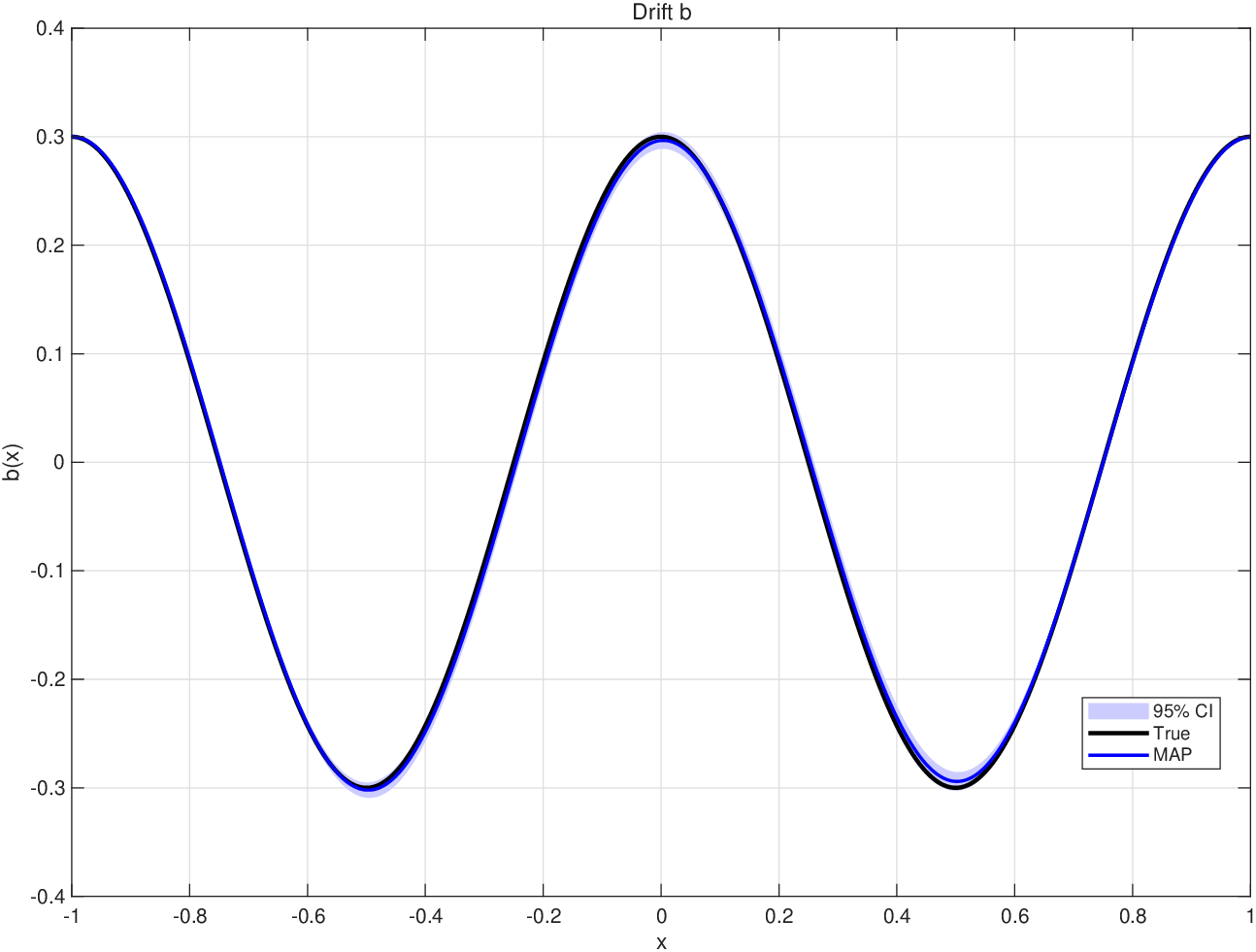}\\
		(d)
		\end{minipage}
		\begin{minipage}{0.3\linewidth}
			\centering
			\includegraphics[height=0.7\columnwidth,width=0.9\linewidth]{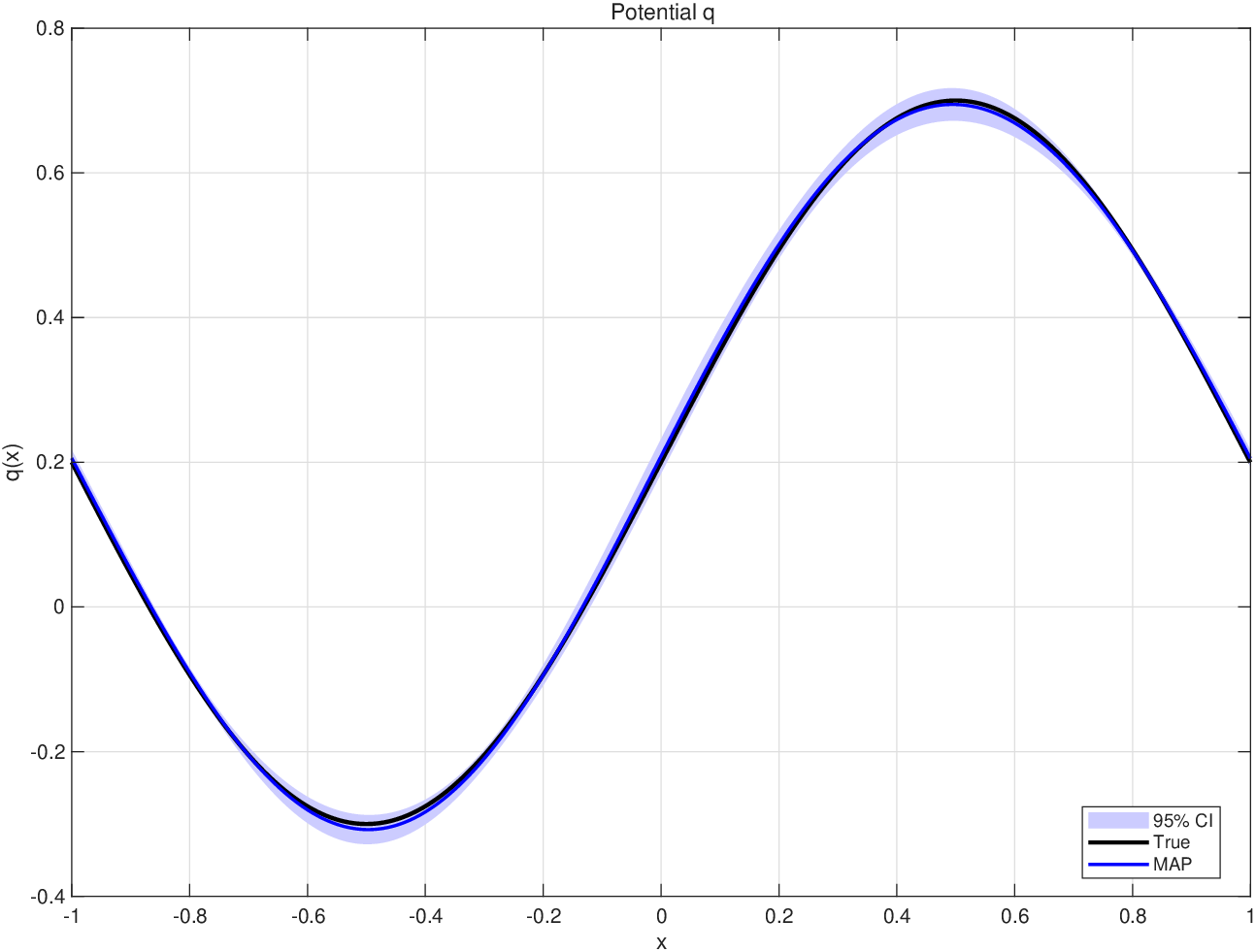}\\
		(e)
		\end{minipage}
		\begin{minipage}{0.3\linewidth}
			\centering
			\includegraphics[height=0.7\columnwidth,width=0.9\linewidth]{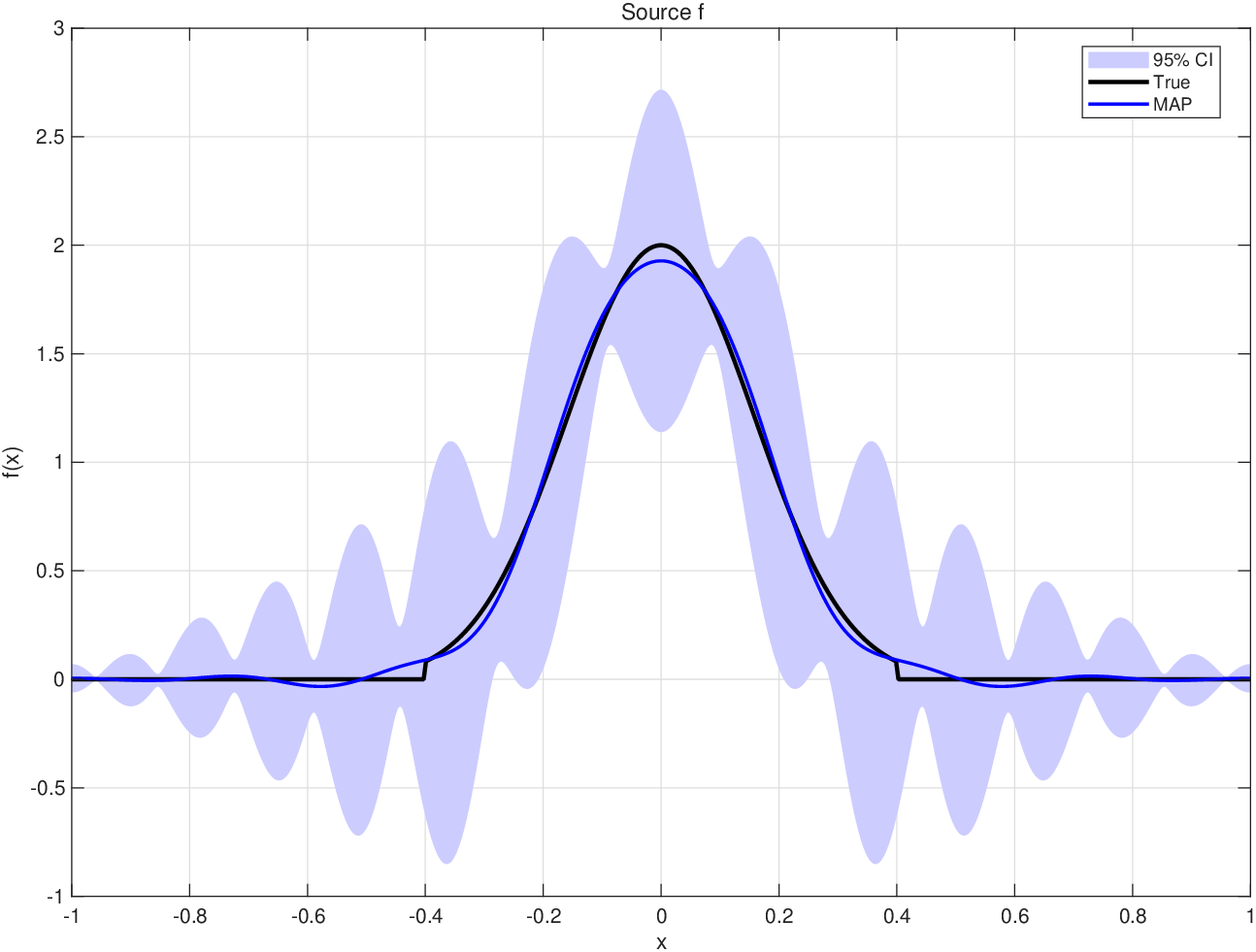}\\
		(f)
		\end{minipage}
	}
\caption{The numerical results for Example 3.}
	\label{fig3}
\end{figure}

\textbf{Example 4:} one-dimensional non-homogeneous case with discontinuous potential and source. Let $\Omega=(-1,1)$ and choose two fractional orders
\begin{equation*}
s_1=0.4,\qquad s_2=0.7.
\end{equation*}
The drift coefficient is
\begin{equation*}
b(x)=0.3\sin(2\pi x),
\end{equation*}
the potential is
\begin{equation*}
q(x)=
\begin{cases}
2, & |x|\le 0.4,\\
0, & \text{otherwise},
\end{cases}
\end{equation*}
and the true source term is
\begin{equation*}
f(x)=
\begin{cases}
-1, & -0.6 < x < -0.2,\\
1.5, & 0.2 < x < 0.6,\\
0, & \text{otherwise}.
\end{cases}
\end{equation*}
The exterior data set is
\begin{equation*}
W_1=(-3,-1-\epsilon)\cup(1+\epsilon,3),\qquad \epsilon=0.05,
\end{equation*}
and we use $J=8$ exterior data functions
\begin{equation*}
g_j(x)=\exp\bigl(-45(x-x_j)^2\bigr),\qquad j=1,\dots,8,
\end{equation*}
with $x_j$ uniformly distributed in $W_1$. The observation set is $W_2=W_1$, and the measurement operator is
\begin{equation*}
\mathcal{O}(u)=\left((-\Delta)^{0.4}u+(-\Delta)^{0.7}u\right)\Big|_{W_2}.
\end{equation*}
For each exterior datum $g_j$, we solve the non-homogeneous problem
\begin{equation*}
\begin{cases}
\bigl((-\Delta)^{0.4}+(-\Delta)^{0.7}+b(x)\partial_x+q(x)\bigr)u_j=f(x), & x\in(-1,1),\\
u_j=g_j, & x\in W_1,
\end{cases}
\end{equation*}
and collect the measurements $\mathcal{O}(u_j)$. In Figure \ref{fig4}, the three figures (a), (b), and (c) respectively show the inversion images of the convective term $b$, the potential term $q$, and the source term $f$ at the error level of  $\delta=10^{-4}$. While figures (d), (e), and (f) respectively present the inversion images at the error level of $\delta=10^{-5}$.
\begin{figure}[htbp]
	\centering
    \captionsetup{font={small}}
    	{
		\begin{minipage}{0.3\linewidth}
			\centering
			\includegraphics[height=0.7\columnwidth,width=0.9\linewidth]{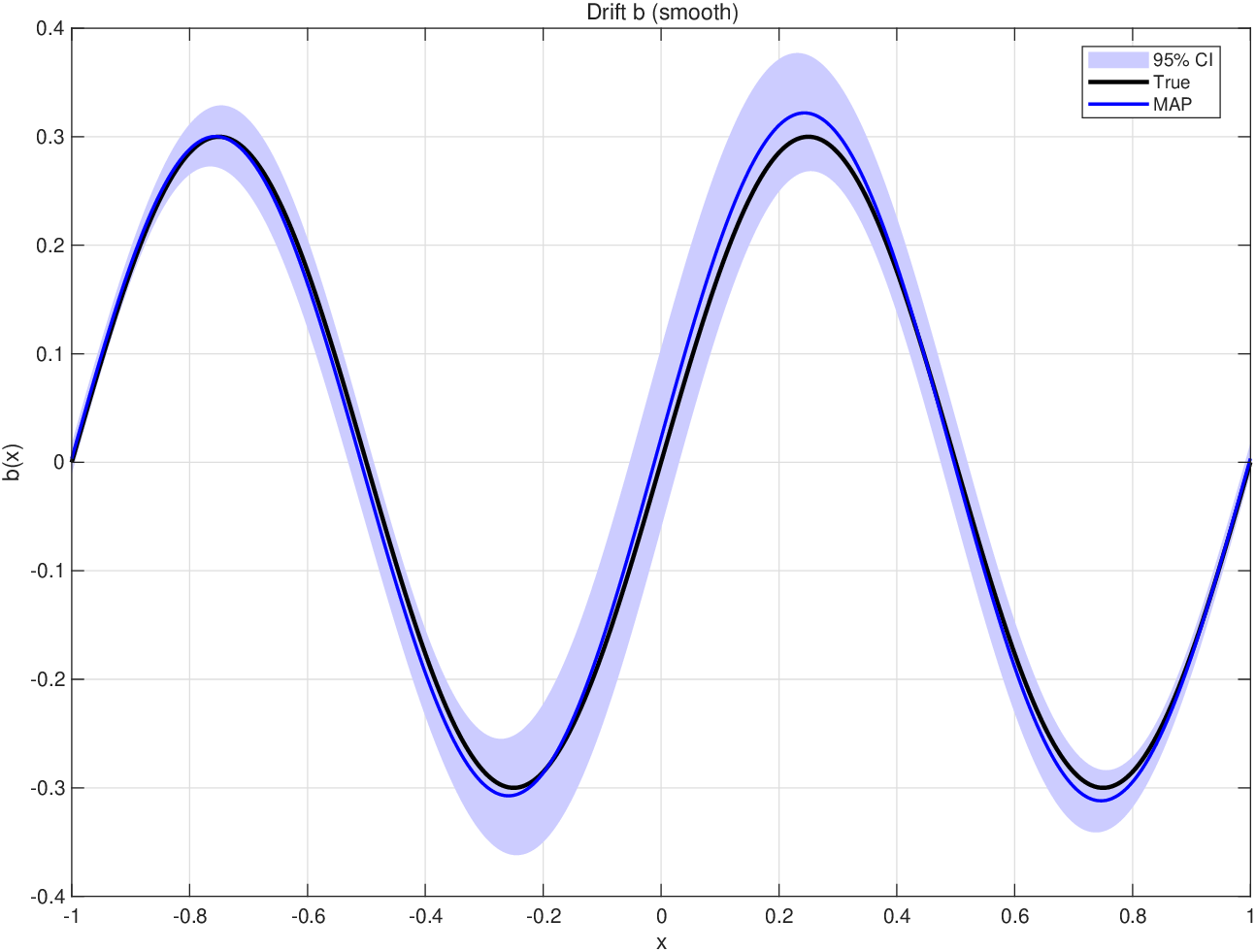}\\
			(a)
		\end{minipage}
		\begin{minipage}{0.3\linewidth}
			\centering
			\includegraphics[height=0.7\columnwidth,width=0.9\linewidth]{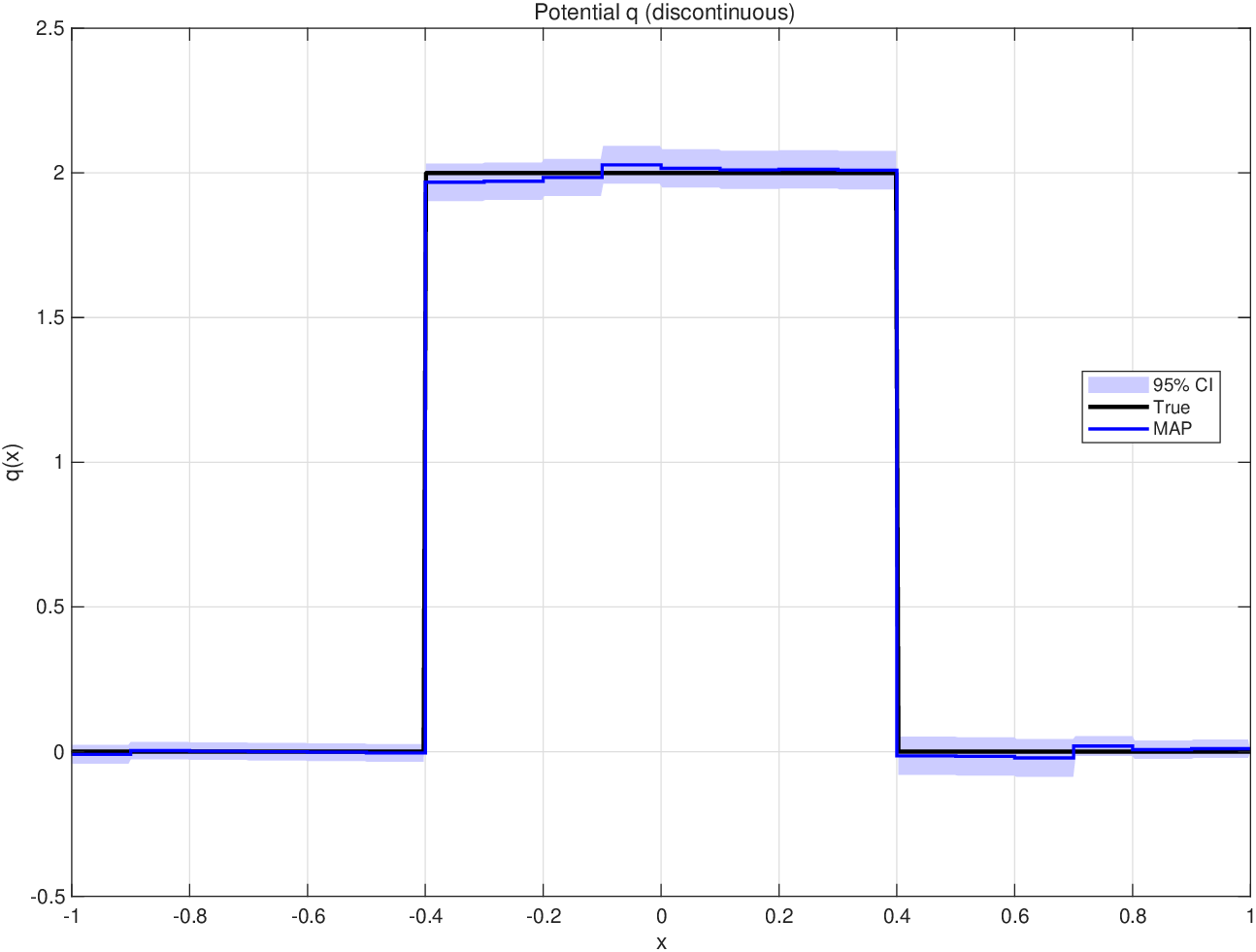}\\
			(b)
		\end{minipage}
		\begin{minipage}{0.3\linewidth}
			\centering
			\includegraphics[height=0.7\columnwidth,width=0.9\linewidth]{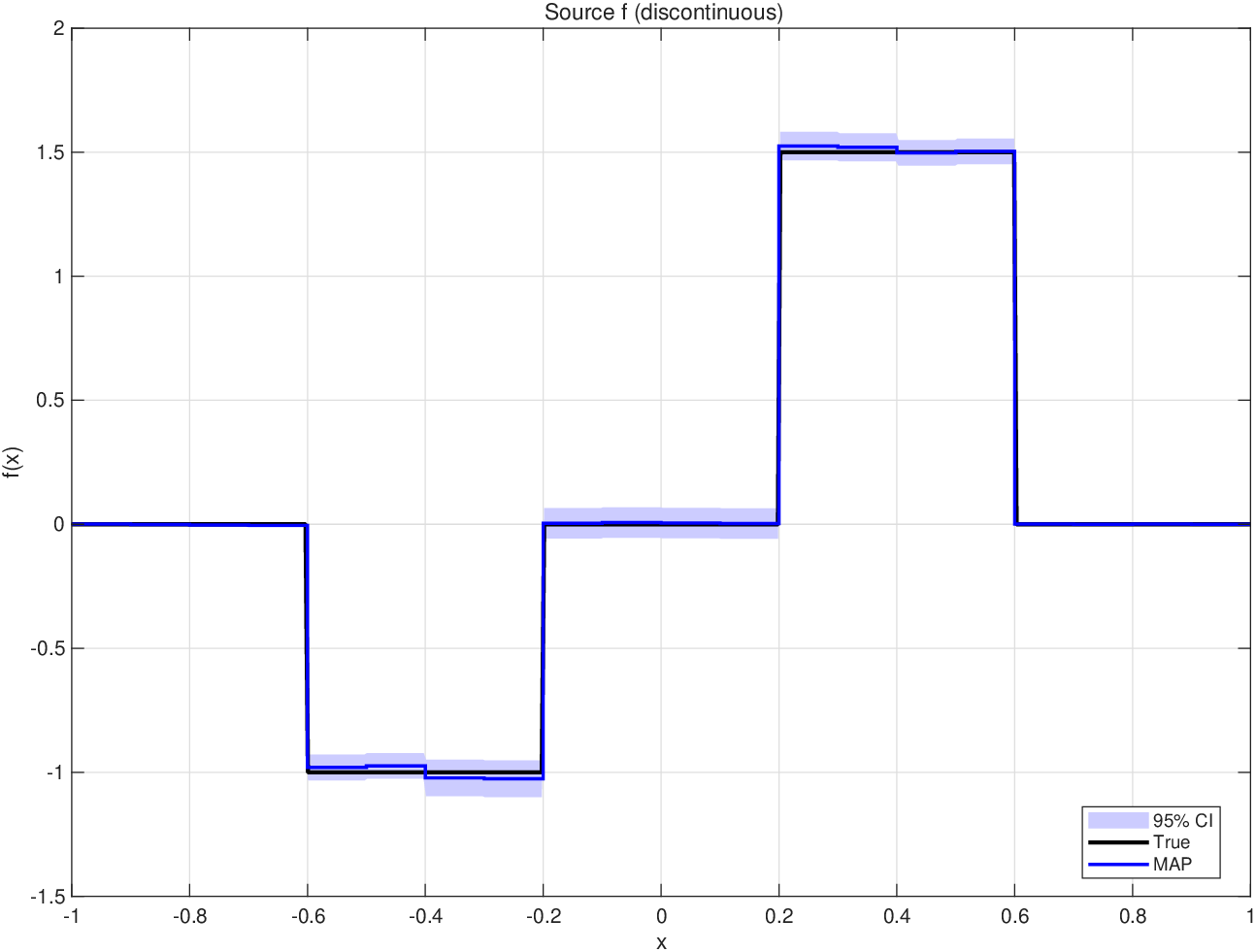}\\
			(c)
		\end{minipage}
	}
\\
	{
		\begin{minipage}{0.3\linewidth}
			\centering
			\includegraphics[height=0.7\columnwidth,width=0.9\linewidth]{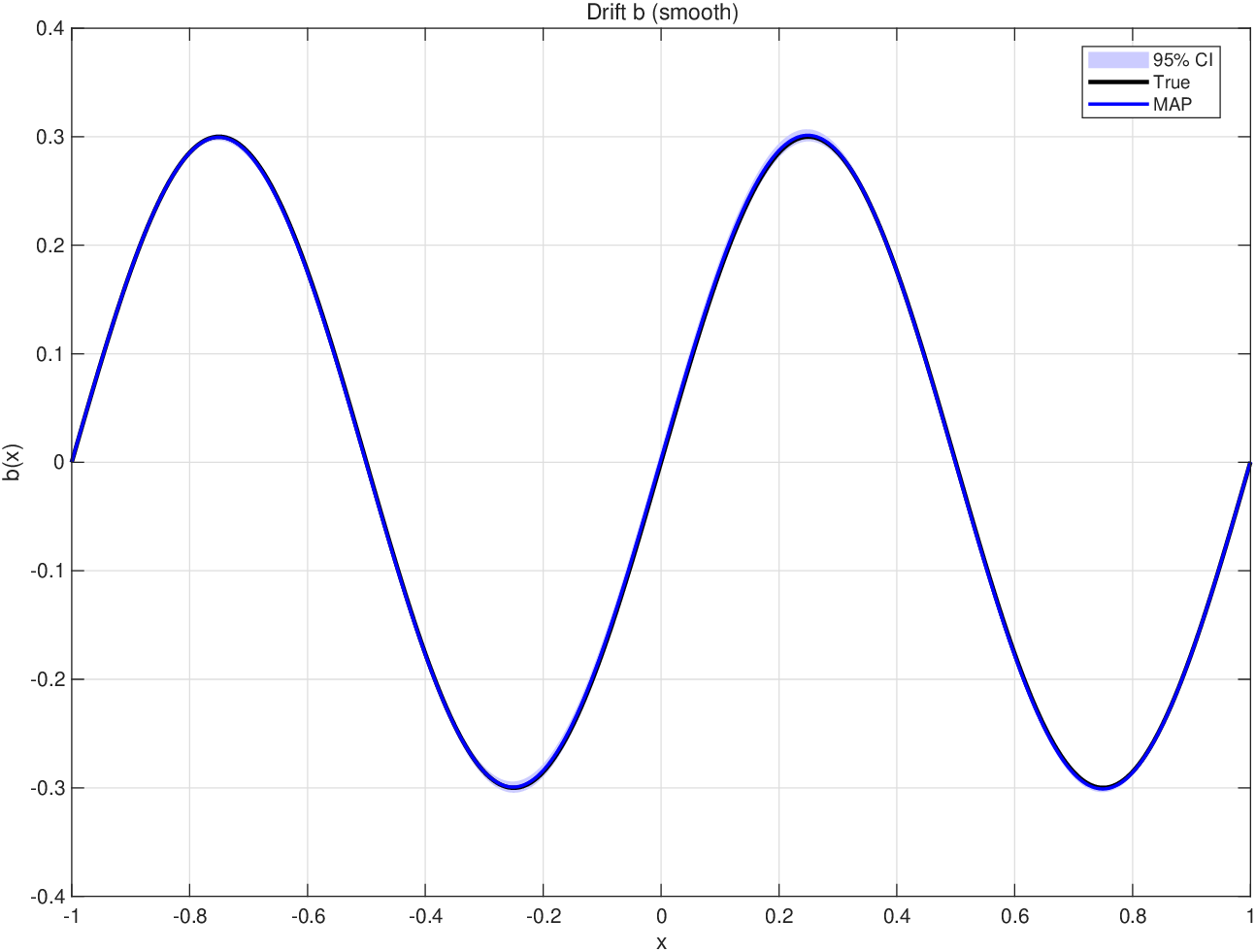}\\
			(d)
		\end{minipage}
		\begin{minipage}{0.3\linewidth}
			\centering
			\includegraphics[height=0.7\columnwidth,width=0.9\linewidth]{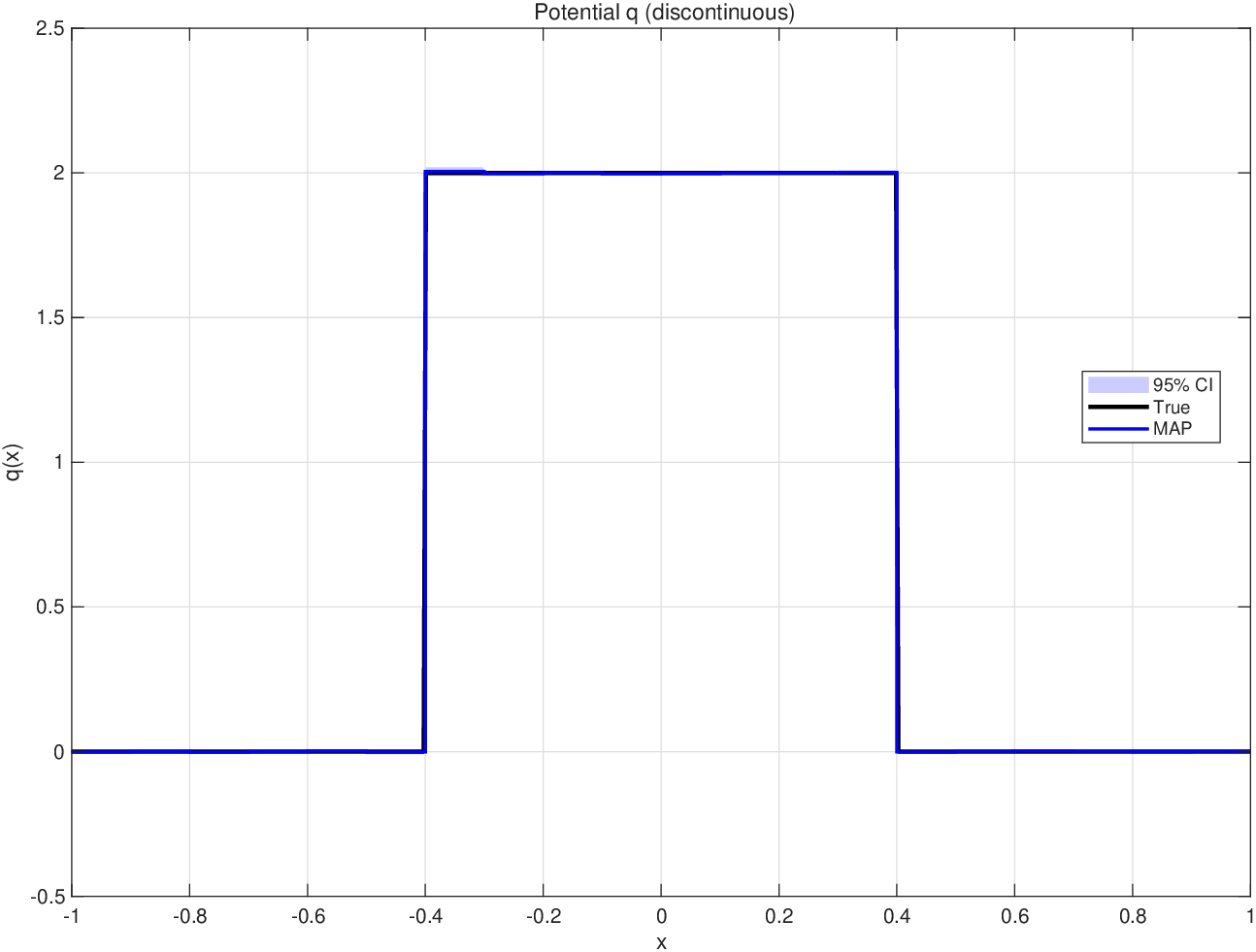}\\
			(e)
		\end{minipage}
		\begin{minipage}{0.3\linewidth}
			\centering
			\includegraphics[height=0.7\columnwidth,width=0.9\linewidth]{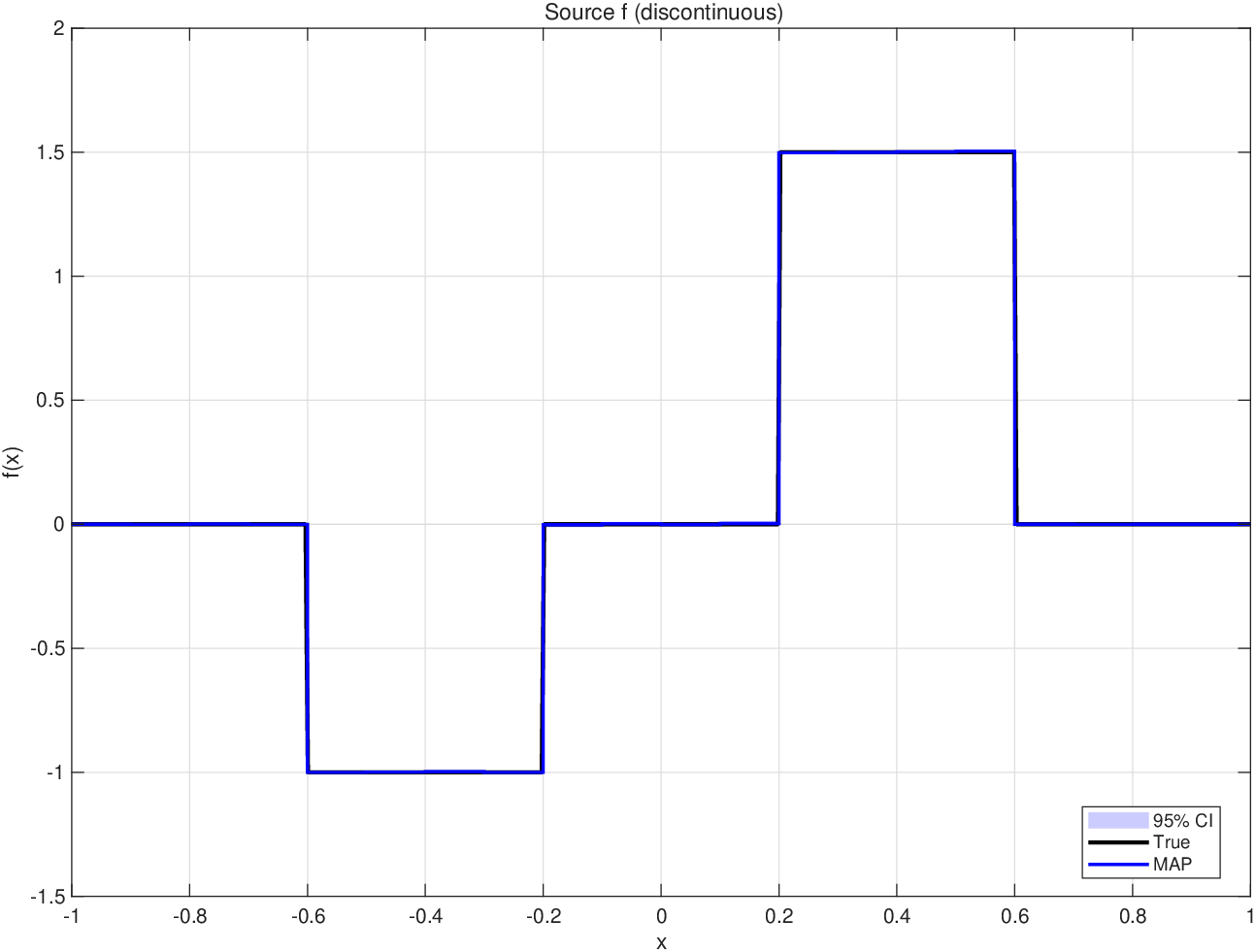}\\
			(f)
		\end{minipage}
	}
\caption{The numerical results for Example 4.}
	\label{fig4}
\end{figure}

From the above examples, it can be seen that the hierarchical Bayesian inversion framework we used is effective. Moreover, by adopting the RBF surrogate model approach, our computational cost has been significantly reduced, greatly improving the computational efficiency. This also proves that our theory on the uniqueness of the simultaneous inversion problem of the multi-fractional Laplacian Schr\"{o}dinger equation with drift terms is correct.
\section{Conclusion}\label{sec7}

We have proposed a unified framework for studying the fractional Calder\'on problem involving drift, potential and source terms for multi-fractional operators
\begin{equation*}
\sum_{k=1}^m (-\Delta)^{s_k} u + b(x)\cdot\nabla u + q(x) u = f(x).
\end{equation*}
In the homogeneous case (\(f=0\)) and the non-homogeneous case (\(f\neq 0\)), we have proved the uniqueness results for simultaneously inverting ($b$, $q$) or ($b$, $q$, $f$). The key elements of these results include the strong unique continuation property of multiple fractional operators, the Runge approximation property, and the generalized Alessandrini identity.

\section*{Appendix}

Proof of Lemma \ref{lm23}:

\begin{proof}
First, Let $0<s_1<\cdots<s_m$ with $s_m>1/2$, and let $q\in L^\infty(\Omega)$. Then for all $u,w\in H^{s_m}(\mathbb R^n)$, we have
\begin{equation}
\label{eq24}
\begin{aligned}
&\left|\sum_{k=1}^m \big((-\Delta)^{s_k/2}v,\,(-\Delta)^{s_k/2}w\big)_{\mathbb R^n}\right|
+ |(q v,w)_\Omega| \\
&\le C(1+\|q\|_{L^\infty(\Omega)})
\Big( \|v\|_{\dot H^{s_m}}\|w\|_{\dot H^{s_m}} + \|v\|_{L^2}\|w\|_{L^2} \Big) \\
&\le C(1+\|q\|_{L^\infty(\Omega)})\,\|v\|_{H^{s_m}}\|w\|_{H^{s_m}},
\end{aligned}
\end{equation}
where the constant $C$ depends only on $n,m,\{s_k\}$.

Next, we will prove \eqref{eq24}, by the triangle inequality and Cauchy-Schwarz,
\begin{equation*}
\left|\sum_{k=1}^m \big((-\Delta)^{s_k/2}v,\,(-\Delta)^{s_k/2}w\big)_{\mathbb R^n}\right|
\le \sum_{k=1}^m \|v\|_{\dot H^{s_k}}\|w\|_{\dot H^{s_k}}.
\end{equation*}
For each $k$, we use a frequency decomposition by splitting the frequency space into $|\xi|\le 1$ and $|\xi|>1$
\begin{equation*}
\|v\|_{\dot H^{s_k}}^2
= \int_{|\xi|\le 1} |\xi|^{2s_k}|\widehat v(\xi)|^2\,d\xi
+ \int_{|\xi|>1} |\xi|^{2s_k}|\widehat v(\xi)|^2\,d\xi.
\end{equation*}
On the low-frequency part $|\xi|\le 1$, we have $|\xi|^{2s_k}\le 1$, hence this contribution is bounded by $\|u\|_{L^2}^2$. On the high-frequency part $|\xi|>1$, we have $|\xi|^{2s_k}\le |\xi|^{2s_m}$, hence the contribution is bounded by $\|u\|_{\dot H^{s_m}}^2$. Therefore,
\begin{equation*}
\|v\|_{\dot H^{s_k}} \le C\big(\|v\|_{L^2}+\|v\|_{\dot H^{s_m}}\big)\le 2C\|v\|_{H^{s_m}},
\end{equation*}
with a constant $C$ depending only on $n,s_k,s_m$, and similarly for $w$. Therefore,
\begin{equation*}
\begin{aligned}
\left|\sum_{k=1}^m \big((-\Delta)^{s_k/2}v,\,(-\Delta)^{s_k/2}w\big)_{\mathbb R^n}\right|
&\le \sum_{k=1}^m \|v\|_{\dot H^{s_k}}\|w\|_{\dot H^{s_k}} \\
&\le C \|v\|_{H^{s_m}}\|w\|_{H^{s_m}}.
\end{aligned}
\end{equation*}
 Since $q\in L^\infty(\Omega)$,
\begin{equation}
\label{eq25}
|(q v,w)_\Omega| \le \|q\|_{L^\infty(\Omega)}\|v\|_{L^2(\Omega)}\|w\|_{L^2(\Omega)}
\le C\|q\|_{L^\infty(\Omega)}\|v\|_{H^{s_m}}\|w\|_{H^{s_m}}.
\end{equation}
Combining \eqref{eq25},
\begin{equation*}
\left|\sum_{k=1}^m \big((-\Delta)^{s_k/2}v,\,(-\Delta)^{s_k/2}w\big)_{\mathbb R^n}\right|
+ |(q v,w)_\Omega|\le C(1+\|q\|_{L^\infty})\|v\|_{H^{s_m}}\|w\|_{H^{s_m}}.
\end{equation*}
Since $0<1-s_m<\frac12$ and $\Omega$ is Lipschitz, we have $\widetilde H^{1-s_m}(\Omega)=H^{1-s_m}(\Omega)$ and the embedding $H^{s_m}(\Omega)\hookrightarrow H^{1-s_m}(\Omega)$ is continuous (because $s_m>1-s_m$). We extend $b$ to a function $B\in W^{1-s_m,\infty}(\mathbb R^n)$ with $\|B\|_{W^{1-s_m,\infty}}\le 2\|b\|_{W^{1-s_m,\infty}(\Omega)}$. Then
\begin{equation}
\label{eq26}
\begin{aligned}
\left|\int_\Omega (b\cdot\nabla v)w\,dx\right|
&\le \|b w\|_{\widetilde H^{1-s_m}(\Omega)} \|\nabla v\|_{H^{s_m-1}(\Omega)} \\
&\le C\|B w\|_{H^{1-s_m}(\mathbb R^n)} \|\nabla v\|_{H^{s_m-1}(\mathbb R^n)} \\
&\le C\|w\|_{H^{1-s_m}(\mathbb R^n)} \|B\|_{W^{1-s_m,\infty}(\mathbb R^n)} \|\nabla v\|_{H^{s_m-1}(\mathbb R^n)} \\
&\le C\|w\|_{H^{s_m}(\mathbb R^n)} \|b\|_{W^{1-s_m,\infty}(\Omega)} \|v\|_{H^{s_m}(\mathbb R^n)}.
\end{aligned}
\end{equation}
In the third line we used the multiplier property of $W^{1-s_m,\infty}$ (which follows from the Kato--Ponce inequality since $1-s_m<1$), and in the last line we used the continuous embedding $H^{s_m}\hookrightarrow H^{1-s_m}$.
Finally, combining \eqref{eq24} and \eqref{eq26}, one has the desired result
\begin{equation}
\label{eq27}
|B_{b,q}(v,w)| \le C\|v\|_{H^{s_m}(\mathbb{R}^n)}\|w\|_{H^{s_m}(\mathbb{R}^n)}.
\end{equation}
\end{proof}

Proof of Lemma \ref{lm24}:

\begin{proof}
First note that, since $s_m$ is the largest exponent among the fractional orders, we have
\begin{equation*}
\sum_{k=1}^m \|(-\Delta)^{s_k/2}u\|_{L^2(\mathbb R^n)}^2 \ge \|(-\Delta)^{s_m/2}u\|_{L^2(\mathbb R^n)}^2.
\end{equation*}
By the fractional Poincar\'e inequality (Lemma 2.1), there exists $C_P>0$ such that
\begin{equation*}
\|u\|_{L^2(\mathbb R^n)} \le C_P \|(-\Delta)^{s_m/2}u\|_{L^2(\mathbb R^n)} \quad \text{for all } u\in \widetilde H^{s_m}(\Omega).
\end{equation*}
Consequently, the norm $\|\cdot\|_{H^{s_m}}$ is equivalent to $\|(-\Delta)^{s_m/2}\cdot\|_{L^2}$ on $\widetilde H^{s_m}(\Omega)$; in particular, there is a constant $C_1>0$ such that $\|u\|_{H^{s_m}} \le C_1 \|(-\Delta)^{s_m/2}u\|_{L^2}$.

Now, from the definition of $B_{b,q}$,
\begin{equation*}
B_{b,q}(u,u) = \sum_{k=1}^m \|(-\Delta)^{s_k/2}u\|_{L^2}^2 + \int_\Omega (b\cdot\nabla u)u\,dx + \int_\Omega q |u|^2\,dx.
\end{equation*}
Thus,
\begin{equation*}
B_{b,q}(u,u) \ge \|(-\Delta)^{s_m/2}u\|_{L^2}^2 - \|q\|_{L^\infty(\Omega)}\|u\|_{L^2}^2 - \left|\int_\Omega (b\cdot\nabla u)u\,dx\right|.
\end{equation*}
For the drift term, we first use the boundedness estimate from Lemma~\ref{lm23} with $w=u$:
\begin{equation*}
\left|\int_\Omega (b\cdot\nabla u)u\,dx\right|
\le C\|b\|_{W^{1-s_m,\infty}(\Omega)}
\|u\|_{H^{1-s_m}(\Omega)}
\|u\|_{H^{s_m}(\Omega)}.
\end{equation*}
Since $1-s_m<s_m$ and $u\in\widetilde H^{s_m}(\Omega)$, the fractional interpolation inequality gives
\begin{equation*}
\|u\|_{H^{1-s_m}(\mathbb R^n)}
\le \|u\|_{L^2(\mathbb R^n)}^{1-\theta}
\|u\|_{H^{s_m}(\mathbb R^n)}^{\theta},
\qquad
\theta=\frac{1-s_m}{s_m}\in(0,1).
\end{equation*}
Hence
\begin{equation*}
\left|\int_\Omega (b\cdot\nabla u)u\,dx\right|
\le C\|b\|_{W^{1-s_m,\infty}(\Omega)}
\|u\|_{L^2(\mathbb R^n)}^{1-\theta}
\|u\|_{H^{s_m}(\mathbb R^n)}^{1+\theta}.
\end{equation*}
Applying Young's inequality with exponents $p=\frac{2}{1+\theta}$ and
$p'=\frac{2}{1-\theta}$ yields
\begin{equation*}
\left|\int_\Omega (b\cdot\nabla u)u\,dx\right|
\le \epsilon\|u\|_{H^{s_m}(\mathbb R^n)}^2
+ C_\epsilon\|b\|_{W^{1-s_m,\infty}(\Omega)}^{\frac{2}{1-\theta}}
\|u\|_{L^2(\mathbb R^n)}^2.
\end{equation*}
Choosing $\epsilon=\frac{1}{4C_1^2}$ so that
$\epsilon\|u\|_{H^{s_m}}^2\le\frac{1}{4}\|(-\Delta)^{s_m/2}u\|_{L^2}^2$,
we obtain
\begin{equation*}
B_{b,q}(u,u) \ge \frac{1}{2}\|(-\Delta)^{s_m/2}u\|_{L^2}^2
- \bigl(\|q\|_{L^\infty}
+ C_\epsilon\|b\|_{W^{1-s_m,\infty}}^{\frac{2}{1-\theta}}\bigr)
\|u\|_{L^2}^2.
\end{equation*}
Using the Poincar¨¦ inequality again to relate
$\|(-\Delta)^{s_m/2}u\|_{L^2}^2$ to $\|u\|_{H^{s_m}}^2$, we conclude
\begin{equation*}
B_{b,q}(u,u) + \mu \|u\|_{L^2}^2 \ge c_0 \|u\|_{H^{s_m}}^2,
\end{equation*}
with $\mu = \|q\|_{L^\infty}
+ C_\epsilon\|b\|_{W^{1-s_m,\infty}}^{\frac{2}{1-\theta}}$
and $c_0>0$ chosen appropriately. This proves the lemma.
\end{proof}

Proof of Lemma \ref{lm25}:
\begin{proof}
We first reduce to the case $g=0$ by setting $v=u-g\in\widetilde H^{s_m}(\Omega)$. The equation becomes
\begin{equation*}
B_{b,q}(v,w)-\lambda(v,w)_\Omega = \widetilde F(w),\quad \forall w\in\widetilde H^{s_m}(\Omega),
\end{equation*}
with $\widetilde F\in(\widetilde H^{s_m})^*$. Hence it suffices to prove the result for $f=0$.

Let $\mu\ge0$ be as in Lemma \ref{lm24} and define $B_\mu(v,w):=B_{b,q}(v,w)+\mu(v,w)_\Omega$. By Lemma \ref{lm24}, $B_\mu$ is coercive on $\widetilde H^{s_m}(\Omega)$:
\begin{equation*}
B_\mu(v,v)\ge c_0\|v\|_{H^{s_m}}^2.
\end{equation*}
By the generalized Lax-Milgram theorem (for continuous coercive bilinear forms), for each $G\in(\widetilde H^{s_m})^*$ there exists a unique $v=G_\mu G\in\widetilde H^{s_m}$ such that
\begin{equation*}
B_\mu(v,w)=G(w),\quad \forall w\in\widetilde H^{s_m},
\end{equation*}
and $\|v\|_{H^{s_m}}\le C\|G\|_{(\widetilde H^{s_m})^*}$. Define $T_\mu:L^2(\Omega)\to L^2(\Omega)$ by $T_\mu h:=G_\mu(Jh)$, where $J:L^2(\Omega)\hookrightarrow(\widetilde H^{s_m})^*$ is the natural embedding. Since $\widetilde H^{s_m}\hookrightarrow L^2$ is compact, $T_\mu$ is compact.

Now the equation $B_{b,q}(v,w)-\lambda(v,w)_\Omega=G(w)$ is equivalent to
\begin{equation}
\label{eq281}
v-(\mu+\lambda)T_\mu v = G_\mu G.
\end{equation}
By the Riesz-Schauder theory for compact operators, the set of non-zero eigenvalues of $T_\mu$ is at most countable and has no finite accumulation point. Define
\begin{equation*}
\Sigma:=\left\{\frac1{\gamma}-\mu : \gamma\in\sigma(T_\mu)\setminus\{0\}\right\}.
\end{equation*}
For $\lambda\notin\Sigma$, $\mu+\lambda\notin\sigma(T_\mu)$, hence $I-(\mu+\lambda)T_\mu$ is invertible on $L^2(\Omega)$. Thus \eqref{eq281} has a unique solution $v\in L^2$, and since $v=(\mu+\lambda)T_\mu v+G_\mu G\in\widetilde H^{s_m}$, we have $v\in\widetilde H^{s_m}$.

For the adjoint problem, the same argument applies to $B_{b,q}^*$; the compact operator $T_\mu^*$ has the same non-zero spectrum as $T_\mu$, hence the same exceptional set $\Sigma$. Finally, if $\lambda=0\notin\Sigma$, then the homogeneous equation has only the trivial solution, which is exactly Assumption \ref{ass1}.
\end{proof}

Proof of Lemma \ref{lm32}:

\begin{proof}
 First, it is easy to see that $\mathcal{D} \subset \widetilde{H}^{s_m}(\Omega)$. By the Hahn-Banach theorem, it suffices to show that for any $F \in (\widetilde{H}^{s_m}(\Omega))^*$ such that $F(v) = 0$ for any $v \in \mathcal{D}$ we must have $F \equiv 0$. Since $F(v) = 0$ for any $v \in \mathcal{D}$, we have
\begin{equation}
\label{eq32}
F(P_{b,q} g - g) = 0 \quad \text{for } g \in C_c^\infty(W).
\end{equation}
Next, we claim that
\begin{equation}
\label{eq33}
F(P_{b,q} g - g) = -B_{b,q}(g,\phi), \quad \text{for } g \in C_c^\infty(W),
\end{equation}
where $\phi \in \widetilde{H}^s(\Omega)$ is the unique solution of
\begin{equation*}
\sum_{k=1}^m (-\Delta)^{s_k} \phi - \nabla \cdot (b\phi) + q\phi = F \quad \text{in } \Omega, \qquad \phi = 0 \text{ in } \Omega_e.
\end{equation*}
In its weak form, it becomes
\begin{equation*}
B_{b,q}^*(\phi,w) = F(w) \quad \text{for } w \in \widetilde{H}^{s_m}(\Omega).
\end{equation*}
Let $g \in C_c^\infty(W)$ and $u_g \coloneqq P_{b,q} g \in H^{s_m}(\mathbb{R}^n)$, then $u_g - g \in \widetilde{H}^{s_m}(\Omega)$ and by \eqref{eq28}
\begin{equation*}
 F(P_{b,q}g - g) = B_{b,q}^*(\phi, u_g - g)=B_{b,q}(P_{b,q}g - g, \phi) = -B_{b,q}(g,\phi),
\end{equation*}
 Combining \eqref{eq32} and \eqref{eq33}, we have
\begin{equation*}
B_{b,q}(g,\phi) = 0 \quad \text{for } g \in C_c^\infty(W).
\end{equation*}
Since $g = 0$ in $\mathbb{R}^n \setminus \overline{W}$ and $\phi \in \widetilde{H}^{s_m}(\Omega)$, $B_{b,q}(g,\phi) = 0$ implies that
\begin{equation*}
0 = \sum_{k=1}^m \int_{\mathbb R^n} (-\Delta)^{s_k/2} g \, (-\Delta)^{s_k/2} \phi \, dx
= \int_{\mathbb R^n} g \, \left(\sum_{k=1}^m (-\Delta)^{s_k}\phi\right) dx.
\end{equation*}
In particular, $\phi \in H^{s_m}(\mathbb{R}^n)$ satisfies
\begin{equation*}
\phi|_W = \sum_{k=1}^m (-\Delta)^{s_k}\phi |_W= 0,
\end{equation*}
 By Lemma \ref{lm26}, \(\phi\equiv 0\), hence \(F\equiv 0\). The Hahn--Banach theorem then implies that \(\mathcal D\) is dense in \(\widetilde H^{s_m}(\Omega)\), and a fortiori in \(L^2(\Omega)\).
\end{proof}

Proof of Lemma \ref{lm34}:
\begin{proof}
First, for any $g\in C_c^\infty(W)$, by definition of the Poisson operator, $P_{b,q}g=g+v$, where $v\in\widetilde H^{s_m}(\Omega)$ satisfies
\begin{equation*}
r_\Omega L_{b,q}v \in C^\infty(\overline{\Omega}).
\end{equation*}
By Proposition~\ref{prop:Hsmu}, we have $v\in\mathcal E^{s_m}(\overline{\Omega})$, which implies $\mathcal D\subset\mathcal E^{s_m}(\overline{\Omega})$.

Next, let $\mathcal L$ be a continuous linear functional on $\mathcal E^{s_m}(\overline{\Omega})$ satisfying
\begin{equation*}
\mathcal L(e_\Omega(r_\Omega P_{b,q}g))=0,\qquad \forall g\in C_c^\infty(W).
\end{equation*}
By the Hahn--Banach theorem, it suffices to show that $\mathcal L\equiv0$. By the definition of the topology of the Fr\'echet space $\mathcal E^{s_m}(\overline{\Omega})$, there exists an integer $\ell$ such that
\begin{equation*}
|\mathcal L(u)|\le C\sum_{k=1}^{\ell}\|u\|_{H^{s_m(k)}(\overline{\Omega})}
\le C\|u\|_{H^{s_m(\ell)}(\overline{\Omega})},\qquad \forall u\in\mathcal E^{s_m}(\overline{\Omega}),
\end{equation*}
for some constant $C>0$ independent of $u$. By Proposition~\ref{prop:Hsmu}(e), $\mathcal E^{s_m}(\overline{\Omega})$ is dense in $H^{s_m(\ell)}(\overline{\Omega})$. Thus $\mathcal L$ has a unique bounded extension $\tilde{\mathcal L}\in (H^{s_m(\ell)}(\overline{\Omega}))^*$.

Consider the homeomorphism from Proposition~\ref{prop:Hsmu}(c),
\begin{equation*}
\mathcal T := r_\Omega L_{b,q} : H^{s_m(\ell)}(\overline{\Omega}) \to H^{\ell-2s_m}(\Omega).
\end{equation*}
Its adjoint is a bounded map
\begin{equation*}
\mathcal T^* : (H^{\ell-2s_m}(\Omega))^* \to (H^{s_m(\ell)}(\overline{\Omega}))^*.
\end{equation*}
The adjoint $\mathcal T^*$ is also a homeomorphism with inverse $(\mathcal T^{-1})^*$. Moreover,
\begin{equation*}
(H^{\ell-2s_m}(\Omega))^* = H_{\overline{\Omega}}^{-\ell+2s_m},
\end{equation*}
so that
\begin{equation*}
\mathcal T^* v(w) = (v,\mathcal T w)_{H_{\overline{\Omega}}^{-\ell+2s_m}\times H^{\ell-2s_m}(\Omega)},\qquad
v\in H_{\overline{\Omega}}^{-\ell+2s_m},\ w\in H^{s_m(\ell)}(\overline{\Omega}).
\end{equation*}
Let $v\in H_{\overline{\Omega}}^{-\ell+2s_m}$ be the unique function satisfying $\mathcal T^*v = \tilde{\mathcal L}$. Choose a sequence $\{v_k\}_{k\in\mathbb N}\subset C_c^\infty(\Omega)$ with $v_k\to v$ in $H^{-\ell+2s_m}(\Omega)$ as $k\to\infty$.

Now, let $g\in C_c^\infty(W)$. Recalling that $e_\Omega(r_\Omega P_{b,q}g)=P_{b,q}g-g$, we compute
\begin{equation}\label{eq:6.1}
\begin{aligned}
0 &= \mathcal L(e_\Omega(r_\Omega P_{b,q}g))
= \tilde{\mathcal L}(P_{b,q}g-g)
= \mathcal T^*v(P_{b,q}g-g) \\
&= (v,\mathcal T(P_{b,q}g-g))
= -(v,\mathcal T g) \\
&= -\lim_{k\to\infty} (v_k, L_{b,q} g) \\
&= -\lim_{k\to\infty} \Bigl( \sum_{j=1}^m (-\Delta)^{s_j} v_k - \nabla\cdot(bv_k) + qv_k,\, g \Bigr),
\end{aligned}
\end{equation}
where we used $\mathcal T P_{b,q}g = 0$ and $v_k\in C_c^\infty(\Omega)$. Since $g\in C_c^\infty(W)$ with $\overline W\cap\overline\Omega=\emptyset$, and $b,q$ are supported in $\Omega$, the last expression simplifies to
\begin{equation*}
\sum_{j=1}^m \bigl((-\Delta)^{s_j} v,\, g\bigr)
= \lim_{k\to\infty} \sum_{j=1}^m \bigl((-\Delta)^{s_j} v_k,\, g\bigr) = 0,
\qquad \forall g\in C_c^\infty(W).
\end{equation*}
Thus $v\in H^{-\ell+2s_m}(\mathbb R^n)$ satisfies
\begin{equation*}
v|_W = \sum_{j=1}^m (-\Delta)^{s_j} v\big|_W = 0.
\end{equation*}
By the strong unique continuation principle for multi-fractional operators (Lemma~\ref{lm26}), we conclude $v\equiv0$ in $\mathbb R^n$. Hence $\tilde{\mathcal L}=0$ and therefore $\mathcal L=0$. This completes the proof.
\end{proof}

\bibliographystyle{unsrt}
\bibliography{liu}

\end{document}